\documentclass[12pt]{amsart}

\usepackage[letterpaper,margin=1in]{geometry}
\usepackage{amssymb,mathtools}
\usepackage{bm}
\usepackage{algorithm}
\usepackage[noend]{algpseudocode}
\usepackage{booktabs}
\usepackage{multirow}
\usepackage{microtype}
\usepackage{float}
\usepackage{tikz}
\usepackage{xcolor}
\usepackage[colorlinks=true,linkcolor=blue!55!black,citecolor=blue!55!black,urlcolor=blue!55!black]{hyperref}
\usepackage[nameinlink,noabbrev]{cleveref}

\newtheorem{theorem}{Theorem}[section]
\newtheorem{lemma}[theorem]{Lemma}
\newtheorem{corollary}[theorem]{Corollary}
\theoremstyle{definition}
\newtheorem{assumption}[theorem]{Assumption}
\theoremstyle{remark}

\numberwithin{equation}{section}
\numberwithin{algorithm}{section}

\newcommand{\R}{\mathbb{R}}
\newcommand{\M}{\underline{\mathbb{M}}}
\newcommand{\K}{\underline{\mathbb{K}}}
\newcommand{\V}{\bm{\mathcal{V}}}
\newcommand{\Z}{\bm{\mathcal{Z}}}
\newcommand{\T}{\mathcal{T}}
\newcommand{\D}{\mathcal{D}}
\newcommand{\F}{\mathcal{F}}
\newcommand{\Acal}{\mathcal{A}}
\newcommand{\as}{\operatorname{as}}
\newcommand{\skw}{\operatorname{skw}}
\newcommand{\tr}{\operatorname{tr}}
\newcommand{\curl}{\operatorname{curl}}
\newcommand{\range}{\operatorname{range}}
\newcommand{\rank}{\operatorname{rank}}
\newcommand{\diag}{\operatorname{diag}}
\newcommand{\jump}[1]{[\![#1]\!]}
\newcommand{\inner}[2]{\left(#1,#2\right)}
\newcommand{\norm}[1]{\left\lVert #1\right\rVert}
\newcommand{\trans}{\mathsf{T}}

\crefname{assumption}{Assumption}{Assumptions}
\Crefname{assumption}{Assumption}{Assumptions}
\crefname{algorithm}{Algorithm}{Algorithms}
\Crefname{algorithm}{Algorithm}{Algorithms}

\begin{document}

\title[Kernel-Aware Two-Grid Preconditioning]{A Kernel-Aware Two-Grid Preconditioner for
Cell-Centered Nearly Incompressible Elasticity}
\author{Shubin Fu}

\address{
Eastern Institute of Technology,
Ningbo, Zhejiang 315200, China
}

\email{sfu@eitech.edu.cn}


\author{Xiang Zhong}

\address{
Department of Mathematics,
The Chinese University of Hong Kong,
Shatin, Hong Kong SAR, China
}
\email{xzhong@math.cuhk.edu.hk}
\date{}
\dedicatory{}

\begin{abstract}
Nearly incompressible elasticity is often discretized by mixed methods to avoid
locking, but the resulting saddle-point systems can be expensive to solve.  We
consider a weakly symmetric multipoint stress discretization whose stress and rotation
unknowns can be eliminated independently over vertex interaction regions, leaving a
symmetric positive definite system for cell-centered displacement only.  Although
smaller, the reduced system contains a parameter-dependent Schur complement that
becomes difficult to precondition as the Lam\'e ratio increases.  We show that its
energy consists of a uniformly coercive shear part and a dominant semidefinite
volumetric part.  Standard cell-centered interpolation need not preserve the
volumetric kernel and can introduce energy amplified near incompressibility.  Motivated by
this structure, we develop a two-grid preconditioner combining two complementary
coarse subspaces: a conventional displacement subspace and a discrete-curl subspace
lying exactly in the fine-grid volumetric kernel.  Symmetric vertex-patch smoothing
completes the displacement-only solver.  The energy splitting and kernel-compatible
decompositions yield a condition-number bound independent of the Lam\'e ratio for
each fixed grid pair.  Two- and three-dimensional experiments confirm the
approximation properties of the locally eliminated discretization and show that the
kernel correction prevents the deterioration of the displacement coarse subspace
alone, also for discontinuous material coefficients.
\end{abstract}

\maketitle

\section{Introduction}

Nearly incompressible linear elasticity poses challenges for both discretization and
iterative solution.  As the Lam\'e ratio \(\lambda/\mu\) grows, volumetric deformation is
increasingly constrained.  An incompatible displacement space may lock, while the
dominant volumetric response may also cause iterative solvers to deteriorate.
Stable approximation and robust solution therefore require complementary treatments.
Mixed stress--displacement formulations
provide stable approximations and retain stress as a primary variable
\cite{arnoldfalkwinther2007,carstensenheuer2025,valseth2021}, but their indefinite systems contain
substantially more globally coupled unknowns than a displacement formulation.  An
efficient treatment should therefore retain mixed stability and stress information
without carrying all variables in the global solve.

This need has motivated the development of multipoint stress methods
\cite{ambartsumyan2020,ambartsumyan2021,yaziciyotov2026}.  Their common idea is to
combine a low-order stress--displacement discretization with vertex localization of
the stress bilinear form.  In mixed finite element constructions, vertex quadrature
makes the stress mass matrix block diagonal over interaction regions, allowing stress
to be eliminated through independent local solves.  The resulting cell-centered
reduction is closely related to the localization underlying multipoint finite volume
methods \cite{nordbotten2015,keilegavlennordbotten2017}.

Within this family, the multipoint stress control volume method of
\cite{fuzhao2025} realizes the same localization without special quadrature.  All
fields are approximated by piecewise constants, and the stress space is designed so
that its bilinear form decomposes into vertex-local contributions.  Stress can
therefore be eliminated locally.  When rotation is constant on each interaction
region, it can be eliminated in a second local step.  The final problem is a sparse
symmetric positive definite system with one displacement vector per cell, while
stress and rotation are recovered by local back-substitution.  This complete
reduction retains local conservation and locking-free approximation in a compact
displacement-only system.

The compact system inherits the conditioning associated with near incompressibility.
As the bulk-to-shear ratio increases, volume-changing displacement components are
penalized much more strongly than nearly volume-preserving components.  In the mixed
formulation, this separation is expressed through the coupling among stress,
displacement, and rotation; local elimination transfers it to the displacement
Schur complement.  Preconditioning the resulting system therefore requires the
discrete volumetric constraint to be identified within the reduced operator.
Related efforts to recover structural information after condensation have also
guided preconditioners for other locally eliminated discretizations
\cite{brunnerkolev2011,dipietro2023}.

Multilevel preconditioning offers a natural route to scalable solution of the
resulting symmetric positive definite system \cite{xuzikatanov2017}, but its effectiveness depends on
preserving this hidden constraint.  Conventional cell-centered interpolation
represents geometrically smooth displacements but need not preserve discrete volume
conservation.  A volume-preserving coarse displacement may therefore acquire
volumetric energy after prolongation, with the resulting defect amplified by the bulk
modulus.  A constraint-preserving coarse correction provides a global treatment of
these modes, complementing the fine-scale correction supplied by local relaxation.

This requirement appears in several established solver frameworks.  Abstract
subspace-correction theory identifies local decomposability of the limiting kernel as
a central condition for parameter-independent convergence
\cite{adler2024,leewuxuzikatanov2007,leewuxuzikatanov2008,wuzheng2014}.  Primal multigrid
methods adapt relaxation and transfer to the divergence constraint
\cite{schoberl1999,farrell2022}, while mixed auxiliary-space methods exploit the
explicit saddle-point structure \cite{chenhuhuang2018}.  Related Schwarz and
substructuring methods obtain parameter robustness through suitable local solvers and
coarse or primal constraints
\cite{beiraolovadinapavarino2006,dohrmannwidlund2009,pavarino2010,
	caipavarinowidlund2015,caipavarino2016,widlund2021}.  Across these frameworks,
parameter robustness rests on decompositions compatible with the limiting constraint.
Transferring this principle to a condensed cell-centered discretization requires a
volume-preserving subspace defined directly from the reduced operator.

The vertex-local balance relations retain the topological information needed for
this construction: they define a discrete measure of
volume change for cell-centered displacements.  The null space of this measure admits
a representation by scalar potentials in two dimensions and edge potentials in three
dimensions.  Related potential representations underlie auxiliary-space
preconditioners in \(H(\operatorname{curl})\) and \(H(\operatorname{div})\)
\cite{hiptmairxu2007,kolevvassilevski2012}.  Prolongating these potentials and then applying the fine-grid discrete
curl produces an exactly volume-preserving displacement transfer.  Based on this
construction, we propose a kernel-aware
two-grid preconditioner with a two-subspace coarse correction.  The first coarse
space is the conventional multilinear
cell-centered displacement space, which represents geometrically smooth error.  The
second is the potential-generated space, whose functions remain in the fine-grid
volume-preserving subspace.  The two corrections are complementary: the first
provides the usual coarse approximation, while the second transfers incompressible
modes without creating artificial volumetric energy.

To complement these global corrections, the interaction regions used for local
elimination provide displacement patches for fine-scale relaxation.  A colored
forward patch sweep, followed by the two Galerkin coarse corrections and a reverse
sweep, defines a symmetric preconditioner.  The prescribed damping ensures positive
definiteness and permits conjugate gradients.  Each coarse correction has its own
Galerkin factorization, while the potentials serve to construct displacement basis
functions.  The resulting iteration therefore retains the cell-centered structure
of the reduced system.

The analysis follows the same separation of volumetric and shear behavior.  Local
Schur-complement identities split the condensed energy into a coercive shear form
and a semidefinite volumetric form.  Exactness of the cell-centered complex then
identifies the volumetric kernel, allowing constrained and general displacements to
be decomposed over the coarse and patch spaces.  Under discrete shear stability,
these decompositions give an additive condition-number bound independent of the
Lam\'e ratio for each fixed grid pair.  An energy comparison carries this bound to
the symmetric composition used by the algorithm.

The numerical results support both the approximation properties and the solver
analysis.  The cell-centered displacement and locally recovered mixed fields retain
their expected convergence rates in the nearly incompressible regime.  In two and
three dimensions, the iteration counts remain controlled as the Lam\'e ratio
spans eight orders of magnitude on each grid, with moderate growth under
refinement.  Removing
the potential-generated correction produces pronounced parameter dependence,
demonstrating the role of the additional coarse space.  Parallel three-dimensional
experiments demonstrate the computational performance on systems with more than
fifty million displacement unknowns, and additional coarse-space comparisons
examine the effect of strongly discontinuous coefficients.

The remainder of the paper is organized as follows.  \Cref{sec:model} presents the
multipoint stress discretization and its reduction to a cell-centered displacement
system, including the three-dimensional local construction.  The two-grid
preconditioner is developed in \cref{sec:preconditioner}, and its convergence
analysis, together with the detailed three-dimensional stability and exact-complex
calculations, is given in \cref{sec:analysis}.  Numerical results, including the
recovered-field, discontinuous-coefficient, and nonhomogeneous-boundary tests, are
reported in \cref{sec:numerics}.

\section{Model problem and cell-centered discretization}
\label{sec:model}
We introduce the weakly symmetric formulation and its multipoint stress
discretization, then eliminate stress and rotation locally to obtain the
cell-centered displacement operator used in \cref{sec:preconditioner}.
\subsection{Weakly symmetric mixed formulation}

Let \(\Omega\subset\R^d\), \(d\in\{2,3\}\), be a bounded Lipschitz
domain.  Let \(\M=\R^{d\times d}\), and let \(\K\subset\M\) denote the
skew-symmetric matrices.  For isotropic elasticity, the compliance operator is
\begin{equation}
  \Acal\underline{\tau}=\frac{1}{2\mu}
  \left(\underline{\tau}-\frac{\lambda}{d\lambda+2\mu}
  \tr(\underline{\tau})\underline{I}\right).
  \label{eq:compliance}
\end{equation}
Here \(\mu>0\) and \(\lambda\geq0\) are the Lam\'e parameters.  Given a
body force \(\bm f\in\bm L^2(\Omega)\), the first-order
stress--displacement--rotation formulation is
\begin{equation}
\begin{aligned}
  \Acal\underline{\sigma}
    &=\nabla\bm u-\underline{\gamma}
      &&\text{in }\Omega,\\
  -\operatorname{div}\underline{\sigma}
    &=\bm f
      &&\text{in }\Omega,\\
  \skw(\underline{\sigma})
    &=0
      &&\text{in }\Omega,\\
  \bm u
    &=\bm 0
      &&\text{on }\partial\Omega.
\end{aligned}
\label{eq:strong-mixed}
\end{equation}
The unknowns are the Cauchy stress \(\underline{\sigma}\), the displacement
\(\bm u\), and the rotation \(\underline{\gamma}\); for a sufficiently
smooth solution, \(\underline{\gamma}=\skw(\nabla\bm u)\).

For
three-dimensional and plane-strain elasticity, Poisson's ratio \(\nu\)
satisfies
\[
  \frac{\lambda}{\mu}=\frac{2\nu}{1-2\nu}.
\]
Thus, \(\nu\to\tfrac12\) corresponds to the nearly incompressible
limit \(\lambda/\mu\to\infty\).

To impose stress symmetry weakly, we introduce the spaces
\[
  \underline{\Sigma}=H(\operatorname{div};\Omega;\M),\qquad
  \bm U=\bm L^2(\Omega),\qquad
  \underline{\Gamma}=L^2(\Omega;\K).
\]
The rotation acts as the Lagrange multiplier for the symmetry constraint.
The weakly symmetric mixed problem is to find
\((\underline{\sigma},\bm u,\underline{\gamma})
\in\underline{\Sigma}\times\bm U\times\underline{\Gamma}\) such that
\begin{align}
  \inner{\Acal\underline{\sigma}}{\underline{\tau}}
  +\inner{\bm u}{\operatorname{div}\underline{\tau}}
  +\inner{\underline{\gamma}}{\underline{\tau}}&=0
       &&\forall\underline{\tau}\in\underline{\Sigma}, \label{eq:mixed1}\\
  -\inner{\operatorname{div}\underline{\sigma}}{\bm v}&=\inner{\bm f}{\bm v}
       &&\forall\bm v\in\bm U, \label{eq:mixed2}\\
  \inner{\underline{\sigma}}{\underline{\eta}}&=0
       &&\forall\underline{\eta}\in\underline{\Gamma}. \label{eq:mixed3}
\end{align}
\subsection{Discrete spaces and mixed scheme}
\label{sec:discretization}

Let \(\T_h^M\) be a conforming partition of \(\Omega\) into rectangles for \(d=2\) or
cuboids for \(d=3\).  Its elements are called macro-elements.  Connecting the center of
each macro-element to its edge midpoints in two dimensions, or to the corresponding
face and edge centers in three dimensions, partitions every macro-element into
\(2^d\) subcells.  The subcell partition is denoted by \(\T_h\).  For every vertex
\(a\) of \(\T_h^M\), the subcells incident to \(a\) form an interaction region
\(\D_a\).

\begin{figure}[H]
  \centering
  \begin{tikzpicture}[scale=1.15,line cap=round,line join=round,
      every node/.style={font=\small}]
    \begin{scope}
      \fill[blue!8] (0,0)--(1.5,0)--(1.5,1.5)--(0,1.5)--cycle;
      \draw[black,thick] (0,0) rectangle (3,3);
      \draw[blue!70!black,densely dashed,line width=1.1pt]
        (1.5,1.5)--(1.5,0)
        (1.5,1.5)--(3,1.5)
        (1.5,1.5)--(1.5,3)
        (1.5,1.5)--(0,1.5);
      \draw[red!75!black,line width=1.4pt]
        (0,0)--(1.5,0) (0,0)--(0,1.5);
      \node[red!75!black,anchor=west] at (0.22,2.55)
        {$\F_h^{\mathrm{pr},1/2}$};
      \draw[red!75!black,->,line width=0.8pt]
        (0.55,2.38)--(0,1.02);
      \node[blue!70!black,anchor=west] at (1.72,2.18)
        {$\F_h^{\mathrm{dl}}$};
      \draw[blue!70!black,->,line width=0.8pt]
        (1.86,2.00)--(1.5,1.82);
      \fill (1.5,1.5) circle (1.8pt)
        node[above right] {$x_M$};
      \fill (0,0) circle (1.8pt) node[below left] {$a$};
      \foreach \x/\y in {1.5/0,3/1.5,1.5/3,0/1.5}
        \fill (\x,\y) circle (1.3pt);
      \node[above left] at (3,3) {$M$};
      \node at (0.72,0.72) {$E$};
      \node[anchor=north] at (1.5,-0.30)
        {\textup{(a)} Macro-element subdivision};
    \end{scope}

    \begin{scope}[shift={(5.0,-0.5)},scale=0.92]
      \fill[blue!8] (1,1) rectangle (3,3);
      \draw[black,thick] (0,0) rectangle (4,4);
      \draw[black,thick] (2,0)--(2,4) (0,2)--(4,2);
      \draw[blue!70!black,densely dashed,line width=1.1pt]
        (1,1)--(1,0) (1,1)--(2,1) (1,1)--(1,2) (1,1)--(0,1)
        (3,1)--(3,0) (3,1)--(4,1) (3,1)--(3,2) (3,1)--(2,1)
        (1,3)--(1,2) (1,3)--(2,3) (1,3)--(1,4) (1,3)--(0,3)
        (3,3)--(3,2) (3,3)--(4,3) (3,3)--(3,4) (3,3)--(2,3);
      \draw[blue!70!black,densely dashed,line width=1.5pt]
        (1,1) rectangle (3,3);
      \draw[red!75!black,line width=1.4pt]
        (2,2)--(2,1) (2,2)--(3,2)
        (2,2)--(2,3) (2,2)--(1,2);
      \foreach \x/\y in {1/1,3/1,1/3,3/3}
        \fill (\x,\y) circle (1.8pt);
      \fill (2,2) circle (2pt) node[above left] {$a$};
      \node at (2.55,2.55) {$\D_a$};
      \node[anchor=north] at (2,-0.35)
        {\textup{(b)} Vertex interaction region};
    \end{scope}

  \end{tikzpicture}
  \caption{Two-dimensional subcell geometry.  The solid red and dashed blue
  segments identify the two classes of subcell interfaces.  The shaded region
  in \textup{(b)} is the union of the four subcells incident to \(a\), namely
  \(\D_a\).}
  \label{fig:subcell-geometry}
\end{figure}
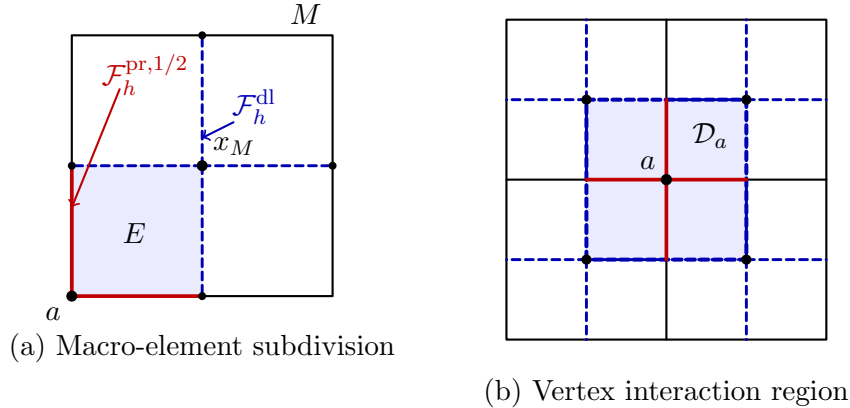

As illustrated in \cref{fig:subcell-geometry},
\(\F_h^{\mathrm{pr},1/2}\) consists of the half-edges inherited from
macro-element boundaries, while \(\F_h^{\mathrm{dl}}\) consists of the interfaces
introduced inside the macro-elements by the subcell partition.  In three dimensions
each macro-element face is divided into four quarter-faces; the same sets denote
these inherited subfaces and the interior subcell faces, respectively.  The discrete
displacement and rotation spaces are
\begin{align}
  \bm U_h&=\{\bm v_h\in\bm L^2(\Omega):
       \bm v_h|_M\in\R^d\quad\forall M\in\T_h^M\}, \label{eq:Uh}\\
  \underline{\Gamma}_h&=\{\underline{\eta}_h\in L^2(\Omega;\K):
       \underline{\eta}_h|_{\D_a}\text{ is constant for every interaction region }\D_a\}.
       \label{eq:Gammah}
\end{align}
Thus each macro-element carries one displacement vector and each interaction region
carries \(d(d-1)/2\) rotation values.

The stress space \(\underline{\Sigma}_h\) consists of matrix fields that are constant on every
subcell, whose row-wise normal components are continuous across
\(\F_h^{\mathrm{pr},1/2}\), and for which no continuity is imposed across
\(\F_h^{\mathrm{dl}}\).  The two-dimensional construction is described in
\cite{fuzhao2025}; the cuboidal basis, local ordering, and interaction-region matrices
used in three dimensions are specified in \cref{sec:three-dimensional-local-construction}.
This construction localizes the compliance form: every stress basis function is
supported on one interaction region, so basis functions associated with distinct
vertices are orthogonal in that form.

For \(\underline{\tau}_h\in\underline{\Sigma}_h\),
\(\bm v_h\in\bm U_h\), and
\(\underline{\eta}_h\in\underline{\Gamma}_h\), define
\begin{align}
  a_h^\sigma(\underline{\tau}_h,\underline{\omega}_h)
    &=\inner{\Acal\underline{\tau}_h}{\underline{\omega}_h}, \label{eq:astress}\\
  b_h(\underline{\tau}_h,\bm v_h)
    &=-\sum_{e\in\F_h^{\mathrm{dl}}}
      \langle \bm v_h,\jump{\underline{\tau}_h\bm n_e}\rangle_e
      =-\sum_{e\in\F_h^{\mathrm{pr},1/2}}
      \langle \underline{\tau}_h\bm n_e,\jump{\bm v_h}\rangle_e, \label{eq:bh}\\
  c_h(\underline{\tau}_h,\underline{\eta}_h)
    &=\inner{\underline{\tau}_h}{\underline{\eta}_h}. \label{eq:ch}
\end{align}
For an interior face, let \(\bm n_e\) point from the minus to the plus side and set
\(\jump{\bm v}_e=\bm v^+-\bm v^-\) and
\(\jump{\underline\tau\bm n_e}_e=(\underline\tau^--\underline\tau^+)\bm n_e\).
On boundary faces the normal is outward and the exterior displacement is zero.
Subcell summation by parts then gives
\(b_h(\underline\tau_h,\bm v_h)
=\sum_{M\in\T_h^M}\langle\underline\tau_h\bm n_M,\bm v_h|_M\rangle_{\partial M}\),
the positive macro-element divergence pairing.  For assembly, write
\(\underline\eta_h=\Xi(\bm\eta_h)\), where the skew-coordinate map is normalized by
\(\underline\tau:\Xi(\bm\eta)=\as(\underline\tau)\cdot\bm\eta\);
thus \eqref{eq:ch} is also \((\as(\underline\tau_h),\bm\eta_h)\).
For later use, equip the cell-centered displacement space with the jump norm
\begin{equation}
  \norm{\bm v_h}_{1,h}^2
  =\sum_{e\in\F_h^{\mathrm{pr},1/2}}
       h_e^{-1}\norm{\jump{\bm v_h}}_{L^2(e)}^2,
  \label{eq:discreteH1norm}
\end{equation}
Here \(h_e\) is the normal distance between the centers of the two macro-elements
adjacent to \(e\).  On a boundary subface it is twice the distance from the
macro-element center to that face, so \(h_e=h\) on a uniform Cartesian grid.
The boundary jump equals the negative interior trace.  Thus
\(\norm{\cdot}_{1,h}\) is a norm under the homogeneous displacement condition.

The discrete mixed problem reads: find
\((\underline{\sigma}_h,\bm u_h,\underline{\gamma}_h)
\in\underline{\Sigma}_h\times\bm U_h\times\underline{\Gamma}_h\) such that
\begin{align}
  a_h^\sigma(\underline{\sigma}_h,\underline{\tau}_h)
      +b_h(\underline{\tau}_h,\bm u_h)
      +c_h(\underline{\tau}_h,\underline{\gamma}_h)&=0
      &&\forall\underline{\tau}_h\in\underline{\Sigma}_h,
      \label{eq:disc1}\\
  b_h(\underline{\sigma}_h,\bm v_h)&=-\inner{\bm f}{\bm v_h}
      &&\forall\bm v_h\in\bm U_h,
      \label{eq:disc2}\\
  c_h(\underline{\sigma}_h,\underline{\eta}_h)&=0
      &&\forall\underline{\eta}_h\in\underline{\Gamma}_h.
      \label{eq:disc3}
\end{align}

\subsection{Local elimination of stress and rotation}

Let \(M_\sigma\), \(B\), and \(C\) be the matrices associated with
\(a_h^\sigma\), \(b_h\), and \(c_h\), respectively.
With \(\boldsymbol f\) denoting the load vector of the physical body force,
\eqref{eq:disc1}--\eqref{eq:disc3} have the symmetric block form
\begin{equation}
  \begin{bmatrix}
    M_\sigma & B^{\trans} & C^{\trans}\\
    B        & 0          & 0\\
    C        & 0          & 0
  \end{bmatrix}
  \begin{bmatrix}\boldsymbol\sigma\\ \boldsymbol u\\ \boldsymbol\gamma\end{bmatrix}
  =
  \begin{bmatrix}0\\ -\boldsymbol f\\ 0\end{bmatrix}.
  \label{eq:fullmatrix}
\end{equation}

The stress matrix is a direct sum over interaction regions,
\begin{equation}
  M_\sigma=\bigoplus_a M_{\sigma,a},
  \label{eq:stressblocks}
\end{equation}
where every \(M_{\sigma,a}\) is symmetric positive definite and has a mesh-independent
size.  Eliminating stress therefore requires only the local inverses
\(M_{\sigma,a}^{-1}\).  Define
\begin{equation*}
\begin{aligned}
  S_{uu}&=BM_\sigma^{-1}B^{\trans},
  &S_{u\gamma}&=BM_\sigma^{-1}C^{\trans},\\
  S_{\gamma u}&=CM_\sigma^{-1}B^{\trans},
  &S_{\gamma\gamma}&=CM_\sigma^{-1}C^{\trans},
\end{aligned}
\end{equation*}
Eliminating stress and multiplying the two resulting equations by \(-1\) gives
\begin{equation}
  \begin{bmatrix}
    S_{uu}&S_{u\gamma}\\
    S_{\gamma u}&S_{\gamma\gamma}
  \end{bmatrix}
  \begin{bmatrix}\boldsymbol u\\\boldsymbol\gamma\end{bmatrix}
  =
  \begin{bmatrix}\boldsymbol f\\0\end{bmatrix}.
  \label{eq:urotation}
\end{equation}
Because the rotation basis is constant on interaction regions,
\(S_{\gamma\gamma}\) is itself a direct sum of
\(d(d-1)/2\)-dimensional vertex blocks.  Rotation is consequently eliminated by a
second collection of local solves.  The final system is
\begin{equation}
\begin{aligned}
  A_h\boldsymbol u&=\boldsymbol b_h,
  &\boldsymbol b_h&=\boldsymbol f,\\
  A_h&=S_{uu}-S_{u\gamma}S_{\gamma\gamma}^{-1}S_{\gamma u}.
\end{aligned}
  \label{eq:reduced}
\end{equation}
Under the local rank condition and the global discrete inf--sup condition stated in
\cref{ass:discrete-stability}, the vertex rotation blocks are invertible and \(A_h\)
is symmetric positive definite.  The global solve thus involves
\(d\,|\T_h^M|\) displacement entries.  Once the displacement is known, rotation
and stress are recovered in sequence through
\(\boldsymbol\gamma=-S_{\gamma\gamma}^{-1}S_{\gamma u}\boldsymbol u\) and
\(\boldsymbol\sigma=-M_\sigma^{-1}(B^{\trans}\boldsymbol u+C^{\trans}\boldsymbol\gamma)\).
Both formulas use the same vertex-local factors as the elimination.  The mixed
problem has therefore been reduced to a positive definite displacement system,
whose parameter-dependent energy determines the preconditioning problem.

For completeness, we next record the three-dimensional local basis, matrix
ordering, and boundary reference patches used by the reduction and by the
stability analysis in \cref{sec:analysis}.

\subsection{Three-dimensional local construction}
\label{sec:three-dimensional-local-construction}
\label{three-dimensional-local-matrices}
The three-dimensional reduction is assembled from the displacement contributions
of individual vertex interaction regions.  We derive each contribution from its
subcell stress basis, first at an interior vertex and then on the boundary
reference patches.  Let \(a\) be an interior vertex of a
uniform cubical macro-mesh of size \(h\).  Its eight incident subcells are indexed by
\begin{equation}
 q(\alpha,\beta,\gamma)=\alpha+2\beta+4\gamma,
 \qquad (\alpha,\beta,\gamma)\in\{0,1\}^3.
 \label{eq:subcell-index}
\end{equation}
The center of the incident macro-element is
\(a+\tfrac h2(2\alpha-1,2\beta-1,2\gamma-1)\).
Each subcell has volume \(h^3/8\), and each primary quarter-face through \(a\) has
area \(h^2/4\).  The local displacement vector is ordered componentwise as
\begin{equation}
 \boldsymbol u_a=
 (u_{0,1},\ldots,u_{7,1},
  u_{0,2},\ldots,u_{7,2},
  u_{0,3},\ldots,u_{7,3})^{\trans}\in\R^{24}.
 \label{eq:u-order}
\end{equation}

For each coordinate direction there are four primary quarter-faces through \(a\),
and every such face carries the three components of the normal traction.  With
\(i=0,1,2\) denoting the traction component, the stress coordinates are
\begin{align}
 r_x(\beta,\gamma,i)&=i+3(\beta+2\gamma),
 &&\beta,\gamma\in\{0,1\}, \label{eq:x-row}\\
 r_y(\alpha,\gamma,i)&=12+i+3(\alpha+2\gamma),
 &&\alpha,\gamma\in\{0,1\}, \label{eq:y-row}\\
 r_z(\alpha,\beta,i)&=24+i+3(\alpha+2\beta),
 &&\alpha,\beta\in\{0,1\}. \label{eq:z-row}
\end{align}
The row ordering identifies a stress basis function with one traction component on
one primary quarter-face, oriented in the positive coordinate direction.
This function is constant on the two adjacent subcells,
has the prescribed continuous normal traction across their common quarter-face, and
has zero normal traction on the remaining primary quarter-faces.  An interior
interaction region consequently has \(4\times3\times3=36\) stress coordinates.

Let \(\{\underline{\Phi}_{a,\ell}\}_{\ell=0}^{n_{\sigma,a}-1}\) be the active local stress basis,
and let \(\bm e_{q,i}\), \(i=0,1,2\), be the constant displacement test vector in
physical component \(i+1\) of incident macro-element \(q\).  These basis functions determine the stress compliance,
displacement coupling, and rotation coupling that enter the local mixed problem.  With
\begin{equation*}
 \as(\underline{\tau})=
 (\underline{\tau}_{32}-\underline{\tau}_{23},\
  \underline{\tau}_{13}-\underline{\tau}_{31},\
  \underline{\tau}_{21}-\underline{\tau}_{12})^{\trans},
\end{equation*}
we use \(\Xi(\bm\eta)\bm v=\bm\eta\times\bm v\), so that
\(\underline\tau:\Xi(\bm\eta)=\as(\underline\tau)\cdot\bm\eta\).
In two dimensions the analogous convention is
\(\Xi(\eta)=\left[\begin{smallmatrix}0&-\eta\\ \eta&0\end{smallmatrix}\right]\)
and \(\as(\underline\tau)=\tau_{21}-\tau_{12}\).
Exact subcell integration defines
\begin{align}
 (M_{\sigma,a})_{\ell m}
 &=\sum_{q\subset\D_a}\frac{h^3}{8}
   \bigl(\Acal_q\underline{\Phi}_{a,m}|_q\bigr):
   \underline{\Phi}_{a,\ell}|_q,
 \label{eq:Msigma}\\
 (B_a)_{(q,i),\ell}
 &=b_a(\underline{\Phi}_{a,\ell},\bm e_{q,i}),
 \label{eq:B}\\
 (C_a)_{r,\ell}
 &=\sum_{q\subset\D_a}\frac{h^3}{8}
   \as(\underline{\Phi}_{a,\ell}|_q)_r.
 \label{eq:C}
\end{align}
For example, an \(x\)-normal quarter-face joining
\(q_-=q(0,\beta,\gamma)\) and \(q_+=q(1,\beta,\gamma)\) contributes
\begin{equation}
 (B_a^{\trans}\boldsymbol u_a)_{r_x(\beta,\gamma,i)}
 =\frac{h^2}{4}\bigl(u_{q_-,i+1}-u_{q_+,i+1}\bigr).
 \label{eq:Bx-entry}
\end{equation}
The opposite outward normals of the two macro-elements give the difference in
\eqref{eq:Bx-entry}, consistently with the macro-element divergence pairing
in \eqref{eq:bh}.  The other coordinate directions follow by cyclic permutation.

With these local couplings, stress elimination produces the displacement--rotation
matrix
\begin{equation}
 \mathcal S_a=
 \begin{bmatrix}B_a\\C_a\end{bmatrix}
 M_{\sigma,a}^{-1}
 \begin{bmatrix}B_a^{\trans}&C_a^{\trans}\end{bmatrix}.
 \label{eq:first-schur}
\end{equation}
The remaining three rotation coordinates are then eliminated to obtain the local
displacement contribution
\begin{equation}
 K_a=B_aM_{\sigma,a}^{-1}B_a^{\trans}
 -B_aM_{\sigma,a}^{-1}C_a^{\trans}
  (C_aM_{\sigma,a}^{-1}C_a^{\trans})^{-1}
  C_aM_{\sigma,a}^{-1}B_a^{\trans}.
 \label{eq:local-K}
\end{equation}
The global displacement matrix is assembled by adding \(K_a\) over the mesh vertices.
Both inverses in \eqref{eq:local-K} act on matrices of uniformly bounded
dimension, so every vertex contribution can be formed independently.

At a boundary vertex, the same indexing is retained and rows supported only on absent
subcells are removed.  This restriction produces the four reference-patch types in
\cref{tab:reference-ranks}.
\begin{table}[htbp]
  \centering
  \caption{Three-dimensional reference interaction regions.}
  \label{tab:reference-ranks}
  \begin{tabular}{lrrrr}
    \toprule
    vertex type & interior & face & edge & corner \\
    \midrule
    incident cells        & 8  & 4  & 2  & 1 \\
    \(n_{\sigma,a}\)      & 36 & 24 & 15 & 9 \\
    \(\rank C_a\)        & 3  & 3  & 3  & 3 \\
    \bottomrule
  \end{tabular}
\end{table}
Since \(M_{\sigma,a}\) is positive definite, the full row rank of \(C_a\) implies that
\(C_aM_{\sigma,a}^{-1}C_a^{\trans}\) is positive definite for every patch type.  Thus
the second local elimination in \eqref{eq:local-K} is well defined at interior and
boundary vertices alike.

\section{A kernel-aware two-grid preconditioner}
\label{sec:preconditioner}

We construct a two-grid preconditioner in the subspace-correction framework
\cite{xu1992}.  Vertex-patch relaxation treats fine-scale coupling, while the coarse
correction combines geometric approximation with preservation of the volumetric
constraint.  We first identify that constraint, then construct the two coarse
subspaces and combine their corrections with symmetric patch sweeps.

\subsection{Discrete volumetric operator and kernel}
\label{sec:kernel}

The discrete volume change at a vertex \(a\) is determined by the normal displacement
jumps across the primary subfaces in \(\D_a\).  On a uniform Cartesian grid of
width \(h\), each such face has measure \((h/2)^{d-1}\).  Collecting their oriented
contributions and normalizing by \(|\D_a|^{1/2}\) defines
\(D_h:\bm U_h\to Q_h^v\), where \(Q_h^v\) contains one scalar value per vertex:
\begin{equation}
  (D_h\bm v_h)_a=\theta_a
  \sum_{M\ni a}\sum_{j=1}^d s_{aM}^{(j)}(\bm v_M)_j,
  \qquad \theta_a=\frac{(h/2)^{d-1}}{\sqrt{|\D_a|}}.
  \label{eq:discretediv}
\end{equation}
Here \(s_{aM}^{(j)}=1\) when the center of \(M\) lies on the negative
\(j\)-side of \(a\), and \(s_{aM}^{(j)}=-1\) on the positive side, consistently
with the divergence pairing \eqref{eq:bh}.  The resulting incompressibility
constraint defines the subspace
\begin{equation}
  \Z_h=\ker(D_h).
  \label{eq:Zh}
\end{equation}
A kernel-preserving transfer can be constructed through discrete potentials.  In two
dimensions let \(\Psi_h^0\) be the scalar potential space on
interior fine-grid vertices and define the cell-centered discrete curl
\[
  K_h=\curl_h:\Psi_h^0\longrightarrow \bm U_h.
\]
In three dimensions let \(\bm E_h^0\) be the lowest-order edge-potential space with the
boundary conditions induced by \eqref{eq:strong-mixed}.  After fixing the gradient gauge,
the cell-centered discrete curl is
\[
  K_h=\curl_h:\bm E_h^0/\nabla_h\Psi_h^0\longrightarrow \bm U_h.
\]
The signed incidence matrices form a discrete complex, and hence
\begin{equation}
  D_hK_h=0.
  \label{eq:divcurl}
\end{equation}
Thus \(\range(K_h)\subseteq\Z_h\).  Any transfer obtained by prolongating potentials
and then applying \(K_h\) consequently has range in the fine-grid incompressible
kernel.

\subsection{Kernel-aware coarse correction}

Let \(\T_H^M\) be a coarse macro-grid obtained by aggregating \(m=H/h\) fine cells in
each coordinate direction, and suppose that the fine and coarse grids are nested.
The global correction combines two coarse subspaces with distinct roles: one
approximates geometrically smooth displacement, while the other transfers
incompressible modes without leaving \(\Z_h\).

For the geometric component, let \(\bm U_H\) contain one vector value at each coarse
cell center.  In one coordinate
direction, interpolate linearly between neighboring coarse cell centers and extend the
first and last coarse values constantly across the two boundary half-strips.  Taking
the tensor product of this one-dimensional rule, component by component, and sampling
at the fine cell centers defines
\begin{equation}
  P_U:\bm U_H\longrightarrow \bm U_h,
  \qquad \V_H^U=\range(P_U).
  \label{eq:Pu}
\end{equation}
The one-dimensional interpolation has full column rank, and therefore its tensor
product \(P_U\) is injective.  Since \(A_h\) is positive definite, the Galerkin matrix
\(A_U=P_U^{\trans}A_hP_U\) is positive definite as well.
Let \(D_H\) be the coarse-grid volumetric operator obtained from
\eqref{eq:discretediv} on \(\T_H^M\).
For a coarse grid with \(n_H\) cells in each direction,
\(\dim(\V_H^U)=d n_H^d\).  This space provides the usual geometric approximation of
smooth displacement fields, but its transfer is not kernel preserving: in general,
\begin{equation}
  D_H\bm z_H=0 \quad\not\Longrightarrow\quad D_hP_U\bm z_H=0.
  \label{eq:Pu-not-kernel-preserving}
\end{equation}

The second component is constructed at the potential level so that its prolongation
remains in \(\Z_h\).  In two dimensions, let
\(Q_\psi:\Psi_H^0\to\Psi_h^0\) be bilinear interpolation of zero-boundary nodal
potentials.  We set
\begin{equation}
  P_{\Z}=K_hQ_\psi,
  \qquad \V_H^{\Z}=\range(P_{\Z}).
  \label{eq:Pk2d}
\end{equation}
If the coarse grid has \(n_H\) cells in each direction, then
\(\dim(\V_H^{\Z})=(n_H-1)^2\).

In three dimensions the degrees of freedom in \(\bm E_H^0\) and \(\bm E_h^0\) are oriented
edge integrals, or equivalently edge cochain values.  Let
\(Q_e:\bm E_H^0\to \bm E_h^0\) be the nested lowest-order edge transfer: along an edge
direction the coarse integral is distributed over the fine subedges, while in the two
transverse directions the coefficients are interpolated bilinearly.  Consequently,
for every oriented coarse edge \(E\),
\begin{equation}
  (\bm e_H)_E=\sum_{e\subset E}(Q_e\bm e_H)_e,
  \qquad Q_e\nabla_H\psi_H=\nabla_hQ_\psi\psi_H,
  \label{eq:edge-transfer-commuting}
\end{equation}
where \(Q_\psi\) in three dimensions is the nested trilinear transfer of
zero-boundary nodal cochains.  The first identity fixes the normalization of the edge
degrees of freedom, and the second is the commuting relation for the discrete gradient; compare the
commuting-transfer construction in \cite{hiptmair1998}.  Gradients, edge transfers,
and the gauge use the common indexing detailed in \cref{sec:exactness}.
Let \(G_H\) inject a gauge-fixed set of coarse edge potentials into
\(\bm E_H^0\).  The kernel transfer is
\begin{equation}
  P_{\Z}=K_hQ_eG_H,
  \qquad \V_H^{\Z}=\range(P_{\Z}).
  \label{eq:Pk3d}
\end{equation}
We use the standard tree--cotree gauge on the relative grid graph: after the boundary
is collapsed to one root, a zero-boundary gradient removes all tree-edge values and
the cotree edges provide independent coordinates.  This removes exactly the gradient
nullspace without changing the curl image.  Hence
\begin{equation}
  \dim(\V_H^{\Z})=2n_H^3-3n_H^2+1.
  \label{eq:kerneldim3d}
\end{equation}
The resulting transfer \(P_{\Z}\) has full column rank, as established in
\cref{cor:kerneltransfer}.
Combining \eqref{eq:divcurl} with either \eqref{eq:Pk2d} or \eqref{eq:Pk3d} gives
\begin{equation}
  D_hP_{\Z}=0.
  \label{eq:kernelpreserve}
\end{equation}
Thus \(\V_H^{\Z}\subset\Z_h\).  Since \(A_h\) is positive definite and \(P_{\Z}\)
has full column rank, its Galerkin matrix \(A_{\Z}=P_{\Z}^{\trans}A_hP_{\Z}\) is
positive definite as well.

The two spaces enter the global correction through separate Galerkin solves:
\begin{equation}
\begin{aligned}
  A_U&=P_U^{\trans}A_hP_U,
  &A_{\Z}&=P_{\Z}^{\trans}A_hP_{\Z},\\
  B_H&=P_UA_U^{-1}P_U^{\trans}
      +P_{\Z}A_{\Z}^{-1}P_{\Z}^{\trans}.
\end{aligned}
  \label{eq:splitcoarse}
\end{equation}
Thus each subspace contributes an independently factored Galerkin correction, and
their sum supplies the global component of the preconditioner.

\subsection{Vertex-patch relaxation}
\label{sec:Vertex-patch}
Fine-scale relaxation complements the two coarse corrections by resolving the
coupling within vertex interaction regions.  Since these regions also determine
the assembly of \(A_h\), we construct the local correction spaces from the cells
surrounding each fine-grid vertex.

Let \(\{\omega_i\}_{i=1}^{N_p}\) be overlapping patches associated with fine-grid
vertices.  The basic patch contains the \(2^d\) cells sharing a vertex, and the analysis
allows any fixed enlargement of this patch.  This construction depends only on patch
topology and can be generated by the same topological principles used in
\cite{farrellknepleymitchellwechsung2021}.  In two dimensions the computations use a
symmetric one-layer enlargement containing at most \(4^2\) cells.  In three dimensions
the basic \(2^3\) patch is enlarged by one cell in the positive coordinate direction,
giving at most \(3^3\) cells.  Let \(R_i:\bm U_h\to \bm U_i\) be restriction to the displacement
unknowns in \(\omega_i\), and define
\begin{equation}
  A_i=R_iA_hR_i^{\trans}.
  \label{eq:patchmatrix}
\end{equation}
Each \(A_i\) is factored once.  Given a residual \(\boldsymbol r\) and a current
correction \(\boldsymbol y\), a forward multiplicative sweep performs
\begin{equation}
  \boldsymbol y\leftarrow\boldsymbol y
  +\omega_sR_i^{\trans}A_i^{-1}R_i(\boldsymbol r-A_h\boldsymbol y),
  \qquad i=1,\ldots,N_p,
  \label{eq:forwardsweep}
\end{equation}
and the backward sweep uses the reverse order.  Here \(\omega_s=1\) for a serial
one-patch-at-a-time sweep. In parallel, the patch indices \(\{1,\ldots,N_p\}\) are partitioned into
\(N_c\) disjoint color classes \(\{\mathcal I_j\}_{j=1}^{N_c}\);
patches within each class are applied simultaneously with a fixed damping
\(\omega_s\), and the colors are traversed forward before the coarse correction
and backward afterward.

Let \(S\) denote the linear residual-to-correction operator generated by one forward
sweep from the zero vector.  Its symmetric multiplicative completion is
\begin{equation}
  \overline S=S+S^{\trans}-S^{\trans}A_hS.
  \label{eq:symmetricsmoother}
\end{equation}

\subsection{Symmetric two-grid composition}

The split coarse correction \(B_H\) acts on global error components, whereas the
forward and backward patch sweeps reduce the remaining fine-scale error.  Combining
them symmetrically gives the two-grid preconditioner
\begin{equation}
  B_{\mathrm{TL}}
  =\overline S+(I-S^{\trans}A_h)B_H(I-A_hS).
  \label{eq:twoleveloperator}
\end{equation}

\begin{algorithm}[H]
\caption{Application of the kernel-aware two-grid preconditioner}
\label{alg:twolevel}
\begin{algorithmic}[1]
\Require Residual \(\boldsymbol r\); fine matrix \(A_h\); patch factors
\(\{A_i^{-1}\}\); coloring \(\{\mathcal I_j\}_{j=1}^{N_c}\); transfers
\(P_U,P_{\Z}\); coarse factors \(A_U^{-1},A_{\Z}^{-1}\).
\State \(\boldsymbol y\gets\boldsymbol 0\)
\For{colors \(j=1,\ldots,N_c\)}
  \State \(\boldsymbol\rho\gets\boldsymbol r-A_h\boldsymbol y\)
  \State \(\boldsymbol y\gets\boldsymbol y+\omega_s\displaystyle\sum_{i\in\mathcal I_j}
  R_i^{\trans}A_i^{-1}R_i\boldsymbol\rho\)
\EndFor
\State \(\boldsymbol r_c\gets\boldsymbol r-A_h\boldsymbol y\)
\State \(\boldsymbol y\gets\boldsymbol y
                   +P_UA_U^{-1}P_U^{\trans}\boldsymbol r_c
                   +P_{\Z}A_{\Z}^{-1}P_{\Z}^{\trans}\boldsymbol r_c\)
\For{colors \(j=N_c,\ldots,1\)}
  \State \(\boldsymbol\rho\gets\boldsymbol r-A_h\boldsymbol y\)
  \State \(\boldsymbol y\gets\boldsymbol y+\omega_s\displaystyle\sum_{i\in\mathcal I_j}
  R_i^{\trans}A_i^{-1}R_i\boldsymbol\rho\)
\EndFor
\State \Return \(\boldsymbol y\)
\end{algorithmic}
\end{algorithm}

\section{Analysis of the preconditioner}
\label{sec:analysis}

Let \(\Omega=(0,1)^d\), with constant Lam\'e coefficients, a uniform \(n^d\)
Cartesian fine grid, and a coarse grid obtained by fixed integer coarsening in each
coordinate direction.  The analysis connects the local elimination to the nearly
singular subspace-correction framework \cite{leewuxuzikatanov2007,xu1992} through
the kernel-compatibility principle \cite{schoberl1999,farrell2022}.  We first separate
the reduced shear and volumetric energies, then use the cell-centered complex to
decompose their error components over the coarse and patch spaces.  Comparing the
additive correction with the symmetric composition completes the condition-number
estimate.  The supporting three-dimensional stability calculation and tensor-product
exactness argument are included in the corresponding subsections below.

\subsection{Parameter structure of the local Schur complements}

Let \(\underline{\Sigma}_{h,a}\) and \(\underline{\Gamma}_{h,a}\) be the stress and rotation spaces on an
interaction region \(\D_a\), and define the weakly symmetric local stress space
\begin{equation}
  \widehat{\underline{\Sigma}}_a
  =\{\underline{\tau}_a\in\underline{\Sigma}_{h,a}:
       c_h(\underline{\tau}_a,\underline{\eta}_a)=0
       \quad\forall\underline{\eta}_a\in\underline{\Gamma}_{h,a}\}.
  \label{eq:localsymmetricstress}
\end{equation}
Write \(b_a\) for the contribution of \(\D_a\) to \(b_h\), and set
\(m_{\lambda,a}(\underline{\tau}_a,\underline{\omega}_a)=
 (\Acal\underline{\tau}_a,\underline{\omega}_a)_{\D_a}\).

Local rotation elimination and control of the reduced shear energy require the
following properties of the mixed spaces.

\begin{assumption}[Local solvability and discrete shear stability]
\label[assumption]{ass:discrete-stability}
For \(d\in\{2,3\}\), each interaction-region rotation matrix \(C_a\) has full row
rank.  In addition, the Cartesian multipoint-stress spaces satisfy
\begin{equation}
 \sup_{0\neq\underline{\tau}_h\in\underline{\Sigma}_h}
 \frac{b_h(\underline{\tau}_h,\bm v_h)+c_h(\underline{\tau}_h,\underline{\eta}_h)}
      {\norm{\underline{\tau}_h}_{L^2(\Omega)}}
 \geq \beta_d
 \left(\norm{\bm v_h}_{1,h}^2+\norm{\underline{\eta}_h}_{L^2(\Omega)}^2\right)^{1/2},
 \label{eq:discrete-infsup}
\end{equation}
where \(\beta_d>0\) is independent of \(h\).
\end{assumption}

The rank condition makes the vertex rotation blocks invertible, while
\eqref{eq:discrete-infsup} supplies the shear-energy lower bound.  The latter is
established for the two-dimensional quadrilateral spaces in
\cite[Lemma~4.2 and Remark~4.1]{fuzhao2025}.  In three dimensions, the reference-patch
ranks are recorded in \cref{sec:three-dimensional-local-construction}; the
generalized-eigenvalue calculation below provides quantitative evidence for the
reduced shear bound, and \cref{lem:mixed-stability-criterion} relates that bound to
\eqref{eq:discrete-infsup}.
\begin{lemma}[Constrained local Schur complement]
\label{lem:localschur}
For \(\bm v_h\in\bm U_h\), let \(\boldsymbol v\) denote its
coefficient vector in the cell-centered displacement basis, and let \(A_h\) be the reduced displacement matrix defined
in \eqref{eq:reduced}.
Then
\begin{equation}
  a_{\lambda,h}(\bm v_h,\bm v_h)
  :=\boldsymbol v^{\trans}A_h\boldsymbol v
  =\sum_a
  \sup_{0\neq\underline{\tau}_a\in
  \widehat{\underline{\Sigma}}_a}
  \frac{b_a(\underline{\tau}_a,\bm v_h)^2}
       {m_{\lambda,a}(\underline{\tau}_a,\underline{\tau}_a)}.
  \label{eq:localschurdual}
\end{equation}
\end{lemma}

\begin{proof}
Let
\(\boldsymbol\ell_a:=B_a^{\trans}\bm v_a
\), where \(\bm v_a\) is the local coefficient vector of \(\bm v_h\) on \(D_a\).
By locally eliminating the stress variable, the contribution associated with \(D_a\) to \(\bm v^{\trans}A_h\bm v\) can be written as
\[
    q_a(\bm v_h)
    :=
    \min_{\boldsymbol\eta_a}
    (\boldsymbol\ell_a+C_a^{\trans}\boldsymbol\eta_a)^{\trans}
    M_{\sigma,a}^{-1}
    (\boldsymbol\ell_a+C_a^{\trans}\boldsymbol\eta_a).
\]
Since \(C_a\) has full row rank,
\(C_aM_{\sigma,a}^{-1}C_a^{\trans}\) is positive definite.  Define
\(\boldsymbol\tau_a^\star
    :=
    -M_{\sigma,a}^{-1}
    (\boldsymbol\ell_a+C_a^{\trans}\boldsymbol\eta_a^\star).
\)
The stationarity condition gives
\(C_a\boldsymbol\tau_a^\star=0,
\)
hence the corresponding stress
\(\underline{\tau}_a^\star\) belongs to
\(\widehat{\underline{\Sigma}}_a\).  Moreover, for every
\(\underline{\omega}_a\in\widehat{\underline{\Sigma}}_a\),
\[
    m_{\lambda,a}
    (\underline{\tau}_a^\star,\underline{\omega}_a)
    =
    -b_a(\underline{\omega}_a,\bm v_h).
\]
Thus \(\underline{\tau}_a^\star\) is the Riesz representative of the
functional
\(\underline{\omega}_a\mapsto-b_a(\underline{\omega}_a,\bm v_h)\).
Therefore, by the Riesz representation theorem,
\[
    q_a(\bm v_h)=m_{\lambda,a}
    (\underline{\tau}_a^\star,\underline{\tau}_a^\star)
    =
    \sup_{0\neq\underline{\omega}_a
    \in\widehat{\underline{\Sigma}}_a}
    \frac{
        b_a(\underline{\omega}_a,\bm v_h)^2
    }{
        m_{\lambda,a}
        (\underline{\omega}_a,\underline{\omega}_a)
    }.
\]
Summing the local contributions over \(a\) yields
\eqref{eq:localschurdual}.
\end{proof}

For later use, denote the local summand in \eqref{eq:localschurdual} by
\begin{equation*}
 a_{\lambda,a}(\bm v_h,\bm v_h)
 =\sup_{0\neq\underline{\tau}_a\in\widehat{\underline{\Sigma}}_a}
       \frac{b_a(\underline{\tau}_a,\bm v_h)^2}
            {m_{\lambda,a}(\underline{\tau}_a,\underline{\tau}_a)},
\end{equation*}
and define its symmetric bilinear version by polarization.  Thus
\(a_{\lambda,h}=\sum_a a_{\lambda,a}\).

Let
\(\operatorname{dev}\underline{\tau}
=\underline{\tau}-d^{-1}\tr(\underline{\tau})\underline I\) and
\(\operatorname{sph}\underline{\tau}
=d^{-1}\tr(\underline{\tau})\underline I\).  The compliance energy has the
orthogonal decomposition
\begin{equation}
  m_{\lambda,a}(\underline{\tau}_a,\underline{\tau}_a)
  =\frac{1}{2\mu}\norm{\operatorname{dev}\underline{\tau}_a}_{\D_a}^2
   +\frac{1}{d\lambda+2\mu}
      \norm{\operatorname{sph}\underline{\tau}_a}_{\D_a}^2.
  \label{eq:compliancesplit}
\end{equation}
Let
\begin{equation}
 \underline J_a=\frac{\underline I}{\sqrt{d|\D_a|}}\quad\hbox{on }\D_a,
 \qquad \norm{\underline J_a}_{L^2(\D_a)}=1.
 \label{eq:normalized-spherical-mode}
\end{equation}
For each coordinate direction, summing the quarter-face basis functions whose
prescribed traction equals the face normal yields the constant field
\(\underline I\) on \(\D_a\) (see \cite{fuzhao2025} and \cref{sec:three-dimensional-local-construction}); hence \(\underline J_a\in\underline{\Sigma}_{h,a}\),
and \(\as(\underline J_a)=0\), so this constant symmetric stress belongs to
\(\widehat{\underline{\Sigma}}_a\).  Put
\begin{equation}
  \widehat{\underline{\Sigma}}_a^0
  =\{\underline{\tau}_a\in\widehat{\underline{\Sigma}}_a:
       (\underline{\tau}_a,\underline J_a)_{\D_a}=0\},
  \qquad d_a(\bm v_h)=b_a(\underline J_a,\bm v_h).
  \label{eq:localstresssplit}
\end{equation}

\begin{lemma}[Uniform separation of the spherical mode]
\label{lem:sphericalsplit}
There is a constant \(C_{\rm loc}\), independent of \(h\), \(\lambda\), and
\(\mu\), such that
\begin{equation}
  a_{\lambda,a}(\bm v_h,\bm v_h)
  =(d\lambda+2\mu)d_a(\bm v_h)^2+r_{\lambda,a}(\bm v_h,\bm v_h),
  \label{eq:localsphericalenergy}
\end{equation}
where
\begin{equation}
 r_{\lambda,a}(\bm v_h,\bm v_h)
 =\sup_{0\neq\underline{\tau}_a^0\in\widehat{\underline{\Sigma}}_a^0}
   \frac{b_a(\underline{\tau}_a^0,\bm v_h)^2}
        {m_{\lambda,a}(\underline{\tau}_a^0,\underline{\tau}_a^0)},
 \label{eq:localremainder}
\end{equation}
and
\begin{equation}
  r_{0,a}(\bm v_h,\bm v_h)
  \leq r_{\lambda,a}(\bm v_h,\bm v_h)
  \leq C_{\rm loc}r_{0,a}(\bm v_h,\bm v_h).
  \label{eq:localremainderbound}
\end{equation}
\end{lemma}

\begin{proof}
If \(\operatorname{dev}\underline{\tau}_a=0\), then
\(\underline{\tau}_a|_E=\alpha_E\underline I\)
on each subcell \(E\subset D_a\).
Normal continuity and the connectivity of the
interaction-region adjacency graph imply that
\(\alpha_E=\alpha_a\) throughout \(D_a\).
Consequently,
\[
    \underline{\tau}_a
    =\alpha_a\underline I
    =\alpha_a\sqrt{d|D_a|}\,\underline J_a
    \qquad\text{on }D_a.
\]
Hence
\(\ker(\operatorname{dev}|_{\widehat{\underline{\Sigma}}_a})
    =\operatorname{span}\{\underline J_a\}.
\)
The decomposition
\(\widehat{\underline{\Sigma}}_a
=\operatorname{span}\{\underline J_a\}
\oplus\widehat{\underline{\Sigma}}_a^0\)
is \(m_{\lambda,a}\)-orthogonal, with
\[
 m_{\lambda,a}(\underline J_a,\underline J_a)
 =\frac{1}{d\lambda+2\mu},
 \qquad
 m_{\lambda,a}(\underline J_a,\underline{\tau}_a^0)=0
 \quad
 \forall\underline{\tau}_a^0\in\widehat{\underline{\Sigma}}_a^0.
\]
Let $\underline{\tau}_a=c\underline J_a+\underline{\tau}_a^0$ with $c\in\mathbb{R}$ and $\underline{\tau}_a^0\in\widehat{\underline{\Sigma}}_a^0$. Combining the above orthogonality and Riesz representation theorem, the squared dual norm in
\eqref{eq:localschurdual} can be split as
\[
\begin{aligned}
a_{\lambda,a}(\bm v_h,\bm v_h)
&=
\sup_{(c,\underline{\tau}_a^0)\neq(0,0)}
\frac{
\left[c\,d_a(\bm v_h)
+b_a(\underline{\tau}_a^0,\bm v_h)\right]^2
}{
\dfrac{c^2}{d\lambda+2\mu}
+m_{\lambda,a}(\underline{\tau}_a^0,\underline{\tau}_a^0)
}=
\frac{b_a(\underline J_a,\bm v_h)^2}
     {m_{\lambda,a}(\underline J_a,\underline J_a)}
+r_{\lambda,a}(\bm v_h,\bm v_h)\\
&=
(d\lambda+2\mu)d_a(\bm v_h)^2
+r_{\lambda,a}(\bm v_h,\bm v_h),
\end{aligned}
\]
which proves the exact splitting
\eqref{eq:localsphericalenergy}.
On
\(\widehat{\underline{\Sigma}}_a^0\), the deviator is injective; finite-dimensional norm equivalence
on the finitely many Cartesian reference patches and affine scaling give
\begin{equation}
  \norm{\underline{\tau}_a}_{\D_a}^2
  \leq C_{\rm dev}\norm{\operatorname{dev}\underline{\tau}_a}_{\D_a}^2
  \qquad\forall\underline{\tau}_a\in\widehat{\underline{\Sigma}}_a^0,
  \label{eq:localdevcoercivity}
\end{equation}
uniformly in the patch type and \(h\).  Formula \eqref{eq:compliancesplit} then
shows \(m_{\lambda,a}\simeq m_{0,a}\) on this complement, and taking inverse dual
norms proves \eqref{eq:localremainderbound}.
\end{proof}

Evaluating \(b_a\) on the normalized spherical stress \(\underline J_a\) relates
the local energy splitting to the volume-change operator introduced in
\cref{sec:kernel}.  The half-face representation \eqref{eq:bh} gives
\begin{equation}
 d_a(\bm v_h)
 =-\frac{1}{\sqrt{d|\D_a|}}
   \sum_{\substack{
       e\in\F_h^{\mathrm{pr},1/2},
       e\subset\D_a
   }}
       |e|\,\bm n_e\cdot\jump{\bm v_h}_e.
 \label{eq:da-half-face}
\end{equation}
Since \(|e|=(h/2)^{d-1}\), collecting the coefficients of each incident cell
recovers the signs and weights in \eqref{eq:discretediv}, yielding
\begin{equation}
  (D_h\bm v_h)_a=\sqrt d\,d_a(\bm v_h).
  \label{eq:Dhfromsphere}
\end{equation}
Thus the volume-change operator used to construct the kernel correction is also
the one generated by the spherical part of the local Schur complement.  To control
the remaining shear contribution, we first bound the stress--displacement coupling
in the jump norm.

\begin{lemma}[Discrete coupling continuity]
\label{lem:bh-continuity}
There is a mesh-independent constant \(C_b\) such that
\begin{equation}
  |b_h(\underline{\tau}_h,\bm v_h)|
  \leq C_b\norm{\underline{\tau}_h}_{L^2(\Omega)}\norm{\bm v_h}_{1,h}
  \qquad\forall (\underline{\tau}_h,\bm v_h)\in\underline{\Sigma}_h\times \bm U_h.
  \label{eq:bh-continuity}
\end{equation}
\end{lemma}

\begin{proof}
The summation-by-parts form in \eqref{eq:bh} and Cauchy--Schwarz give
\begin{equation*}
 |b_h(\underline{\tau}_h,\bm v_h)|
 \leq
 \left(\sum_{e\in\F_h^{\mathrm{pr},1/2}}
        h_e\norm{\underline{\tau}_h \bm n_e}_{L^2(e)}^2\right)^{1/2}
 \norm{\bm v_h}_{1,h}.
\end{equation*}
On each reference subcell, the normal-trace degrees of freedom are bounded by the
subcell \(L^2\) norm.  Affine scaling to a Cartesian subcell therefore yields
\begin{equation*}
 \sum_{e\in\F_h^{\mathrm{pr},1/2}}
 h_e\norm{\underline{\tau}_h \bm n_e}_{L^2(e)}^2
 \leq C\norm{\underline{\tau}_h}_{L^2(\Omega)}^2.
\end{equation*}
Every subcell face has uniformly bounded multiplicity in this sum, which proves
\eqref{eq:bh-continuity}.
\end{proof}

In three dimensions, the mixed shear estimate can be related to the assembled
displacement energy through the block structure of the locally eliminated system.
Let \(M_{0,h}\) and \(A_{0,h}\) denote the stress and reduced displacement
matrices at \(\lambda=0\).  The following criterion identifies the lower energy
bound needed for \eqref{eq:discrete-infsup}.

\begin{lemma}[Block criterion for mixed shear stability]
\label{lem:mixed-stability-criterion}
Let \(M_{\Gamma,h}\) be the \(L^2\) mass matrix on \(\underline{\Gamma}_h\), and
let \(J_h\) be the matrix associated
with the jump norm, i.e.,
\(\bm v_h^{\trans}J_h\bm v_h
    =\|\bm v_h\|_{1,h}^2\) for all \(\bm v_h\in\bm U_h\).  Suppose that
the reference-patch matrices \(C_a\) have full row rank and that
\begin{equation}
 \bm v_h^{\trans}A_{0,h}\bm v_h
 \geq c_0\mu \bm v_h^{\trans}J_h\bm v_h
 \qquad\forall \bm v_h\in \bm U_h,
 \label{eq:reduced-shear-lower-criterion}
\end{equation}
with \(c_0\) independent of \(h\).  Then the mixed estimate
\eqref{eq:discrete-infsup} holds with a mesh-independent constant.
\end{lemma}

\begin{proof}
To relate the reduced bound to mixed stability, suppose
\(A_{0,h}\succeq c_0\mu J_h\) with \(c_0>0\) independent of \(h\).  The
displacement--rotation Gram matrix at \(\lambda=0\) is
\begin{equation*}
 \mathcal K_{0,h}
 =\begin{bmatrix}B_h\\ C_h\end{bmatrix}M_{0,h}^{-1}
  \begin{bmatrix}B_h^{\trans}&C_h^{\trans}\end{bmatrix}
 =\begin{bmatrix}S_{uu}&S_{u\gamma}\\
                   S_{\gamma u}&S_{\gamma\gamma}\end{bmatrix}.
\end{equation*}
Reference-patch rank and affine scaling imply
\begin{equation}
 c_\gamma\mu M_{\Gamma,h}
 \preceq S_{\gamma\gamma}
 \preceq C_\gamma\mu M_{\Gamma,h},
 \label{eq:rotation-equivalence}
\end{equation}
where the constants are independent of \(h\).  Since the displacement Schur
complement is \(A_{0,h}\), setting
\(T_h=S_{\gamma\gamma}^{-1}S_{\gamma u}\) gives
\begin{equation}
 \mathcal K_{0,h}
 =\begin{bmatrix}I&T_h^{\trans}\\0&I\end{bmatrix}
  \begin{bmatrix}A_{0,h}&0\\0&S_{\gamma\gamma}\end{bmatrix}
  \begin{bmatrix}I&0\\T_h&I\end{bmatrix}.
 \label{eq:mixed-factorization}
\end{equation}
The Schur-complement ordering and continuity of the stress--displacement coupling
yield
\begin{equation*}
 \norm{T_h\bm v_h}_{S_{\gamma\gamma}}^2
 \leq \bm v_h^{\trans}S_{uu}\bm v_h
 \leq C\mu\norm{\bm v_h}_{1,h}^2.
\end{equation*}
Thus the triangular map
\((\bm v_h,\underline{\eta}_h)
\mapsto(\bm v_h,\underline{\eta}_h+T_h\bm v_h)\) and its inverse are uniformly bounded
in the displacement--rotation product norm.  Combining
\eqref{eq:rotation-equivalence}, \eqref{eq:mixed-factorization}, and
\(A_{0,h}\succeq c_0\mu J_h\) gives
\begin{equation*}
 \begin{bmatrix}\bm v_h\\\underline{\eta}_h\end{bmatrix}^{\trans}
 \mathcal K_{0,h}
 \begin{bmatrix}\bm v_h\\\underline{\eta}_h\end{bmatrix}
 \geq c\mu\bigl(\norm{\bm v_h}_{1,h}^2+
                    \norm{\underline{\eta}_h}_{L^2(\Omega)}^2\bigr).
\end{equation*}
Since \(M_{0,h}=(2\mu)^{-1}M_{L^2,h}\), the quadratic form on the left is
\(2\mu\) times the square of the mixed supremum with the stress \(L^2\) norm.
Dividing by \(2\mu\) therefore gives \eqref{eq:discrete-infsup}.
\end{proof}

The reduced bound in \eqref{eq:reduced-shear-lower-criterion} is
characterized by the smallest generalized eigenvalue
\begin{equation}
 \vartheta_h=\min_{0\neq\bm v_h}
 \frac{\bm v_h^{\trans}A_{0,h}\bm v_h}
      {\mu\bm v_h^{\trans}J_h\bm v_h}.
 \label{eq:shear-eigenvalue}
\end{equation}
The values in \cref{tab:shear-eigenvalues} give quantitative evidence of
shear-energy control in the jump norm: they decrease from 1.423 to 1.022
under refinement.  The shift-and-invert eigensolves have a maximum relative
eigenpair residual of \(5.3\times10^{-11}\).

\begin{table}[H]
  \centering
  \caption{Smallest generalized eigenvalue in
  \eqref{eq:shear-eigenvalue} on uniform cubical grids.}
  \label{tab:shear-eigenvalues}
  \begin{tabular}{lrrrrrr}
    \toprule
    grid & \(4^3\) & \(6^3\) & \(8^3\) & \(12^3\) & \(16^3\) & \(20^3\) \\
    \midrule
    \(\vartheta_h\) & 1.423 & 1.219 & 1.132 & 1.061 & 1.034 & 1.022 \\
    \bottomrule
  \end{tabular}
\end{table}

With the shear stability specified in \cref{ass:discrete-stability},
the local spherical splitting now yields a global characterization of the
reduced energy.

\begin{theorem}[Uniform reduced-operator equivalence]
\label{thm:reduced-equivalence}
Under \cref{ass:discrete-stability}, the reduced form satisfies
\begin{equation}
  a_{0,h}(\bm v_h,\bm v_h)+\lambda\norm{D_h\bm v_h}_2^2
  \leq a_{\lambda,h}(\bm v_h,\bm v_h)
  \leq C_*\bigl(a_{0,h}(\bm v_h,\bm v_h)
                 +\lambda\norm{D_h\bm v_h}_2^2\bigr),
  \label{eq:reduced-equivalence}
\end{equation}
where \(C_*\) is independent of \(h\) and \(\lambda/\mu\).  Moreover,
\begin{equation}
  a_{0,h}(\bm v_h,\bm v_h)\simeq\mu\norm{\bm v_h}_{1,h}^2.
  \label{eq:shear-jump-equivalence}
\end{equation}
\end{theorem}

\begin{proof}
Summing \eqref{eq:localsphericalenergy} and using
\eqref{eq:localremainderbound} proves \eqref{eq:reduced-equivalence}.  To prove
\eqref{eq:shear-jump-equivalence}, the Schur-complement
characterization used in the proof of
\cref{lem:localschur} gives
\begin{equation*}
 a_{0,h}(\bm v_h,\bm v_h)
 =\min_{\underline{\eta}_h\in\underline{\Gamma}_h}\sup_{0\neq\underline{\tau}_h\in\underline{\Sigma}_h}
 \frac{\bigl(b_h(\underline{\tau}_h,\bm v_h)+c_h(\underline{\tau}_h,\underline{\eta}_h)\bigr)^2}
      {(\Acal_0\underline{\tau}_h,\underline{\tau}_h)}.
\end{equation*}
Since the inf-sup estimate in \cref{ass:discrete-stability} holds for every
\(\underline{\eta}_h\in\underline{\Gamma}_h\), taking the minimum gives
\(a_{0,h}(\bm v_h,\bm v_h)\gtrsim\mu\|\bm v_h\|_{1,h}^2\).
  For the reverse bound, choose \(\underline{\eta}_h=0\) in the minimum, use
\cref{lem:bh-continuity}, and note that
\((\Acal_0\underline{\tau}_h,\underline{\tau}_h)=(2\mu)^{-1}\norm{\underline{\tau}_h}_{L^2(\Omega)}^2\).
\end{proof}

To apply nearly singular subspace-correction estimates, we retain the exact
parameter dependence supplied by \cref{lem:sphericalsplit}.  Define
\begin{equation}
 q_h(\bm v_h,\bm w_h)=(D_h\bm v_h)^{\trans}D_h\bm w_h.
 \label{eq:qh}
\end{equation}
We further define the bilinear form
\(s_{\lambda,h}\) by
\begin{equation}
 \mu s_{\lambda,h}(\bm v_h,\bm w_h)
 =\frac{2\mu}{d}q_h(\bm v_h,\bm w_h)+\sum_a r_{\lambda,a}(\bm v_h,\bm w_h).
 \label{eq:slambdadefinition}
\end{equation}
Then the reduced form has the exact representation
\begin{equation}
 \frac{1}{\mu}a_{\lambda,h}
 =s_{\lambda,h}+tq_h,
 \qquad t=\lambda/\mu.
 \label{eq:exactnearlysingularform}
\end{equation}
Indeed, the spherical contribution in \eqref{eq:localsphericalenergy} is
\((\lambda+2\mu/d)\norm{D_h\bm v_h}_2^2\).  Moreover, setting
\(s_h=\mu^{-1}a_{0,h}\), the remainder bounds in \eqref{eq:localremainderbound} yield
\begin{equation}
 c_s s_h(\bm v_h,\bm v_h)
 \leq s_{\lambda,h}(\bm v_h,\bm v_h)
 \leq C_s s_h(\bm v_h,\bm v_h),
 \label{eq:slambdaequivalence}
\end{equation}
with constants independent of \(h\) and \(t\).  These bounds allow decompositions
in the fixed shear norm \(s_h\) to control the parameter-dependent form
\(s_{\lambda,h}\) in \eqref{eq:exactnearlysingularform}.

\subsection{Exactness of the cell-centered complex}
\label{sec:exactness}

The volumetric form in \eqref{eq:exactnearlysingularform} singles out
\(\ker D_h\) as the limiting constraint space.  To decompose this space using the
potential transfers \eqref{eq:Pk2d} and \eqref{eq:Pk3d}, we establish exactness of
the cell-centered complex.  Let \(G,M\in\R^{(n+1)\times n}\) be the signed and unsigned
one-dimensional vertex--cell incidence matrices.  Both have rank \(n\), and their
left nullspaces are generated by the constant and alternating vectors.  Hence
\begin{equation}
  \dim\bigl(\range(G)\cap\range(M)\bigr)=n-1.
  \label{eq:GMintersection}
\end{equation}
With \(\Theta_h=\diag(\theta_a)\) from \eqref{eq:discretediv}, the unweighted
incidence \(\widehat D_h=\Theta_h^{-1}D_h\) has the componentwise tensor-product form
\begin{align}
  \widehat D_h&=[G\otimes M,\;M\otimes G] &&(d=2),
  \label{eq:Dkron2d}\\
  \widehat D_h&=[G\otimes M\otimes M,\;M\otimes G\otimes M,\;
        M\otimes M\otimes G] &&(d=3).
  \label{eq:Dkron3d}
\end{align}
Since \(\Theta_h\) is invertible, these blocks determine \(\ker D_h\).
We now record the discrete curl realization, the cancellation identity, and the
transfer compatibility used in the exactness argument.

In two dimensions, a zero-boundary vertex potential \(\psi_{ij}\) is mapped to
the cell-centered vector field
\begin{align*}
 (K_h\psi)_{ij,1}
 &=\frac{\psi_{i-1,j}+\psi_{i,j}-\psi_{i-1,j-1}-\psi_{i,j-1}}{2h},\\
 (K_h\psi)_{ij,2}
 &=-\frac{\psi_{i,j-1}+\psi_{i,j}-\psi_{i-1,j-1}-\psi_{i-1,j}}{2h}.
\end{align*}
The two incidence contributions cancel directly, giving \(\widehat D_hK_h=0\)
and hence \(D_hK_h=0\).

For three dimensions, index cells by \(j=0,\ldots,n-1\) and vertices by
\(k=0,\ldots,n\).  Fix
\begin{align*}
 G_{k,j}&=\delta_{k,j+1}-\delta_{k,j},
 &M_{k,j}&=\delta_{k,j+1}+\delta_{k,j},\\
 R_{k\ell}&=\delta_{k,n-\ell},
 &A&=\tfrac12M^{\trans}R,\quad L=G^{\trans}R.
\end{align*}
The vertex coordinates of a potential are reversed in all directions.  An
\(x\)-edge cochain retains increasing order along \(x\) and reverses its two
transverse vertex indices; the \(y\)- and \(z\)-components follow cyclically.
Edge orientations remain the positive physical coordinate directions.  With full
arrays understood before imposing the relative boundary conditions, the matching
gradient is
\[
 \nabla_h\psi=
 \begin{bmatrix}
 L\otimes I_{n+1}\otimes I_{n+1}\\
 I_{n+1}\otimes L\otimes I_{n+1}\\
 I_{n+1}\otimes I_{n+1}\otimes L
 \end{bmatrix}\psi.
\]
Nodal potentials vanish on the boundary, and edge cochains vanish whenever a
transverse vertex index is a boundary index.  A direct computation gives
\begin{equation}
 GA=-\tfrac12ML
 \qquad\hbox{on vectors whose first and last entries vanish.}
 \label{eq:GA-identity}
\end{equation}
Indeed, \(GA+ML/2\) is supported only at \((0,n)\) and \((n,0)\).
For oriented edge integrals, the cell-centered curl is
\begin{align*}
 (K_he)_x&=h^{-2}\bigl[(A\otimes L\otimes I_n)e_z-(A\otimes I_n\otimes L)e_y\bigr],\\
 (K_he)_y&=h^{-2}\bigl[(I_n\otimes A\otimes L)e_x-(L\otimes A\otimes I_n)e_z\bigr],\\
 (K_he)_z&=h^{-2}\bigl[(L\otimes I_n\otimes A)e_y-(I_n\otimes L\otimes A)e_x\bigr].
\end{align*}
The factor \(h^{-2}\) converts face circulations to displacement values.
Applying \eqref{eq:GA-identity} in the transverse directions cancels the
incidence contributions, so \(\widehat D_hK_h=0\) and \(D_hK_h=0\).
The matching gradient gives \(K_h\nabla_h=0\) by the same cancellation of mixed
differences.

The same coordinate ordering determines the coarse transfers.  Define the
vertex and edge permutation matrices
\begin{align*}
 \mathcal R_{0,h}&=R\otimes R\otimes R,\\
 \mathcal R_{1,h}&=\operatorname{diag}
 (I_n\otimes R\otimes R,\ R\otimes I_n\otimes R,\ R\otimes R\otimes I_n).
\end{align*}
If a superscript \({\rm std}\) denotes increasing physical-coordinate ordering,
then
\begin{align*}
 \nabla_h&=\mathcal R_{1,h}^{\trans}\nabla_h^{\rm std}\mathcal R_{0,h},\\
 Q_e&=\mathcal R_{1,h}^{\trans}Q_e^{\rm std}\mathcal R_{1,H},
 &Q_\psi&=\mathcal R_{0,h}^{\trans}Q_\psi^{\rm std}\mathcal R_{0,H}.
\end{align*}
These permutations preserve the relative boundary conditions, so coarse-edge
summation and \(Q_e\nabla_H=\nabla_hQ_\psi\) carry over to this ordering.
Using the same edge labels in the tree--cotree gauge then makes the gauge
compatible with \(\nabla_H\).

\begin{lemma}[Exact discrete potential representation]
\label{lem:exactcomplex}
On an \(n^d\) Cartesian grid with homogeneous displacement data,
\begin{align}
  \dim\ker(D_h)&=(n-1)^2 &&(d=2),
  \label{eq:kernel-dim-2d}\\
  \dim\ker(D_h)&=2n^3-3n^2+1 &&(d=3).
  \label{eq:kernel-dim-3d}
\end{align}
Furthermore,
\begin{equation}
  \range(K_h)=\ker(D_h).
  \label{eq:exactcomplex}
\end{equation}
\end{lemma}
\begin{proof}
The inclusion \(\range(K_h)\subseteq\ker(D_h)\) becomes equality once their
dimensions agree.  Since \(\rank(D_h)=\rank(\widehat D_h)\), the dimension of
the volumetric kernel can be computed from the tensor-product blocks of
\(\widehat D_h\).  Let
\(X=\range(G)\cap\range(M)\).  Since the left nullspaces of \(G\) and \(M\)
are generated by the constant and alternating vectors, respectively,
\(\dim X=n-1\).  In two dimensions, the ranges of the two blocks of \(\widehat D_h\)
intersect in \(X\otimes X\) (the intersection of two tensor products of
subspaces is the tensor product of the intersections), and hence
\begin{equation*}
 \rank(D_h)=2n^2-(n-1)^2=n^2+2n-1,
 \qquad \dim\ker(D_h)=(n-1)^2.
\end{equation*}
In three dimensions, a common direct-sum decomposition of the block ranges
allows the same dimension calculation.  Fix
\(g\in\range(G)\setminus X\) and \(m\in\range(M)\setminus X\); a vector of
\(\range(G)\) lying in \(\range(M)\) belongs to \(X\), so
\(\R^{n+1}=X\oplus\operatorname{span}\{g\}\oplus\operatorname{span}\{m\}\).
Consequently \((\R^{n+1})^{\otimes3}\) is the direct sum of the \(27\) type
subspaces \(Y_1\otimes Y_2\otimes Y_3\) with
\(Y_i\in\{X,\operatorname{span}\{g\},\operatorname{span}\{m\}\}\), and each
block range, for example
\(\range(G)\otimes\range(M)\otimes\range(M)\) with
\(\range(G)=X\oplus\operatorname{span}\{g\}\) and
\(\range(M)=X\oplus\operatorname{span}\{m\}\), is the direct sum of the type
subspaces it contains.  Sums and intersections of the three block ranges
therefore correspond to unions and intersections of their index sets, for which
inclusion--exclusion is valid.  In particular the pairwise intersections have
dimension \(n(n-1)^2\), the triple intersection is \(X^{\otimes3}\), and
counting the type subspaces contained in at least one block range gives
\begin{equation*}
 \rank(D_h)=3n^3-3n(n-1)^2+(n-1)^3=n^3+3n^2-1,
\end{equation*}
and \(\dim\ker(D_h)=2n^3-3n^2+1\).

It remains to compute the curl rank.  In two dimensions, \(K_h\psi=0\) equates
opposite diagonal values on each cell; propagation
to the zero boundary gives \(\psi=0\), so
\(\rank(K_h)=(n-1)^2\).  In three dimensions, extend a relative edge cochain by
zero to boundary edges and undo the above permutations.  In each coordinate
direction, the cell-centered curl averages the circulations on the two parallel elementary faces of a cell, so a
vanishing discrete curl forces consecutive face circulations along every grid
line to sum to zero.  For a relative cochain every face contained in
\(\partial\Omega\) has zero circulation, and a vector \(y\) with \(y_0=0\) and
\(y_{i-1}+y_i=0\) for all \(i\) vanishes; hence every elementary face has zero
circulation.  Because these face boundaries generate the cycle space of a
cubical box, path integration from a boundary vertex gives
\begin{equation*}
 \ker(K_h)=\nabla_h\Psi_h^0,
 \qquad
 \rank(K_h)=3n(n-1)^2-(n-1)^3=2n^3-3n^2+1.
\end{equation*}
The cancellation identity and the matching dimensions give
\(\range(K_h)=\ker(D_h)\) in both dimensions, which is the exact potential
representation required by the kernel coarse correction.
\end{proof}

\begin{corollary}[Uniform kernel transfer]
\label{cor:kerneltransfer}
The transfers \eqref{eq:Pk2d} and \eqref{eq:Pk3d} have full column rank and
satisfy \(D_hP_{\Z}=0\), and
\begin{equation}
  a_{\lambda,h}(P_{\Z}\bm z_H,P_{\Z}\bm z_H)
  \simeq a_{0,h}(P_{\Z}\bm z_H,P_{\Z}\bm z_H)
  \label{eq:coarsekernelenergy}
\end{equation}
uniformly in \(h\) and \(\lambda/\mu\).
\end{corollary}

\begin{proof}
Both transfers have the form \(P_{\Z}=K_hQ\), so
\(\range(P_{\Z})\subset\range(K_h)=\ker(D_h)\).

For injectivity, suppose \(K_hQ\bm z_H=\bm 0\).  In two dimensions \(K_h\) is
injective on zero-boundary potentials by the proof of \cref{lem:exactcomplex},
and \(Q=Q_\psi\) is injective, so \(\bm z_H=\bm 0\).  In three dimensions
\(Q\bm z_H=Q_eG_H\bm z_H\in\ker(K_h)=\nabla_h\Psi_h^0\), say
\(Q_eG_H\bm z_H=\nabla_h\psi_h\).  Summing over the fine subedges of an oriented
coarse edge, the fine gradient telescopes, and the first identity in
\eqref{eq:edge-transfer-commuting} shows that \(G_H\bm z_H\) is the coarse
gradient of the restriction of \(\psi_h\) to the coarse vertices.  This gradient
vanishes on all tree edges, so its potential is constant along the spanning tree
of the relative grid graph and therefore zero, giving \(G_H\bm z_H=\bm 0\) and
\(\bm z_H=\bm 0\).

The estimate \eqref{eq:coarsekernelenergy} follows from
\cref{thm:reduced-equivalence}.
\end{proof}

\subsection{Decompositions over the two coarse spaces and patches}

The inclusion \(\V_H^{\Z}\subset\Z_h\) makes the kernel coarse energy uniformly
equivalent to shear energy.  The next step is to distribute displacement errors
over the coarse and patch spaces with bounded total energy.  We first construct
a shear-stable decomposition of general displacements, then a decomposition
whose components remain in \(\Z_h\).

Let \(\V_i=R_i^{\trans}\bm U_i\) be the fine vertex-patch spaces, with the
coarsening ratio \(m=H/h\) and patch enlargement fixed under refinement.

\begin{lemma}[Shear-stable two-grid decomposition]
\label{lem:shear-decomposition}
For every \(\bm v_h\in \bm U_h\), there are \(\bm v_U\in\V_H^U\) and
\(\bm v_i\in\V_i\) such that
\begin{equation}
  \bm v_h=\bm v_U+\sum_{i=1}^{N_p}\bm v_i,
  \qquad
  s_h(\bm v_U,\bm v_U)+\sum_{i=1}^{N_p}s_h(\bm v_i,\bm v_i)
  \leq C_U s_h(\bm v_h,\bm v_h).
  \label{eq:shear-stable-decomposition}
\end{equation}
The constant \(C_U\) depends on \(d\), \(m\), and the fixed patch enlargement, but
not on \(h\).
\end{lemma}

\begin{proof}
By \eqref{eq:shear-jump-equivalence} and the definition of $s_h$, it is enough to work with the jump norm
\eqref{eq:discreteH1norm}. For a coarse cell \(T\), let \(\omega_T\) be the union of \(T\) and the coarse
cells sharing at least one vertex with \(T\).  Let
\begin{equation*}
 (\Pi_H\bm v_h)_T=m^{-d}\sum_{M\subset T}(\bm v_h)_M,
 \qquad I_H^U\bm v_h=P_U\Pi_H\bm v_h.
\end{equation*}
By the tensor-product construction
of \(P_U\), the restriction \(I_H^U\bm v_h|_T\) depends only on
\(\bm v_h|_{\omega_T}\), each \(\omega_T\) is connected, and only finitely many
reference configurations occur.  Let \(\norm{\cdot}_{1,h,T}\) count the fine
primary-face jumps incident to cells in \(T\), including physical-boundary terms.
All interpolation values in these jumps depend only on \(\omega_T\).
On each reference neighborhood, finite-dimensional boundedness gives
\begin{equation}
 \norm{I_H^U\bm v_h}_{1,h,T}^2+
 H^{-2}\norm{\bm v_h-I_H^U\bm v_h}_{L^2(T)}^2
 \leq C_{d,m}\sum_{e\subset\omega_T}
 h_e^{-1}\norm{\jump{\bm v_h}}_{L^2(e)}^2.
 \label{eq:reference-aggregate-estimate}
\end{equation}
Here the right-hand sum runs over faces internal to \(\omega_T\) and the
faces on \(\partial\omega_T\cap\partial\Omega\).
If it vanishes, connectedness makes \(\bm v_h\) constant on \(\omega_T\); that
constant is zero whenever the neighborhood meets \(\partial\Omega\).
Constant reproduction by \(P_U\Pi_H\) therefore makes the left-hand side zero as
well.  Thus the right-hand kernel is contained in the left-hand kernel, and the
estimate follows on the finite-dimensional quotient.  Both sides scale as
\(H^{d-2}\); summing over the finitely overlapping neighborhoods gives
\begin{equation}
 \norm{I_H^U\bm v_h}_{1,h}\leq C_I\norm{\bm v_h}_{1,h},
 \qquad
 \norm{\bm v_h-I_H^U\bm v_h}_{L^2(\Omega)}
 \leq C_AH\norm{\bm v_h}_{1,h}.
 \label{eq:cell-quasi-interpolation}
\end{equation}
The constants depend only on \(d\) and \(m\); the reference configurations include the finitely
many face, edge, and corner aggregate types.

Put \(\bm e_h=\bm v_h-I_H^U\bm v_h\).  Assign every fine cell to its lexicographically first
incident vertex, and let \(\bm v_i\) be the corresponding part of \(\bm e_h\).  Then
\(\bm e_h=\sum_i \bm v_i\), \(\bm v_i\in\V_i\), and the cellwise inverse estimate gives
\begin{equation*}
 \sum_i\norm{\bm v_i}_{1,h}^2
 \leq C h^{-2}\norm{\bm e_h}_{L^2(\Omega)}^2
 \leq C(H/h)^2\norm{\bm v_h}_{1,h}^2.
\end{equation*}
Together with \eqref{eq:shear-jump-equivalence}, this proves
\eqref{eq:shear-stable-decomposition}.
\end{proof}

\begin{lemma}[Exact local decomposition of the incompressible kernel]
\label{lem:kernel-decomposition}
Define
\begin{equation}
 C_{\Z}(h,H)=
 \sup_{0\neq \bm z_h\in\Z_h}\ \inf_{\substack{
 \bm z_h=\bm z_{\Z}+\sum_i \bm z_i\\
 \bm z_{\Z}\in\V_H^{\Z},\ \bm z_i\in\V_i\cap\Z_h}}
 \frac{s_h(\bm z_{\Z},\bm z_{\Z})+
       \sum_i s_h(\bm z_i,\bm z_i)}{s_h(\bm z_h,\bm z_h)}.
 \label{eq:kernel-decomposition-constant}
\end{equation}
Then \(C_{\Z}(h,H)<\infty\) in both dimensions.  Adding the coarse kernel space
cannot increase this constant.
\end{lemma}

\begin{proof}
By \cref{lem:exactcomplex}, write \(\bm z_h=K_h\psi_h\).  In two dimensions use the
zero-boundary nodal basis; in three dimensions use the tree--cotree gauge following
\eqref{eq:Pk3d}.  Every selected basis curl is supported on the cells incident to one
vertex or edge and therefore belongs to some \(\V_i\cap\Z_h\).  Assigning each curl
to such a patch gives an exact decomposition with \(\bm z_{\Z}=\bm 0\).

Let $\boldsymbol{x}$
be the coefficient vector of $\boldsymbol{z}_h$
with respect to the basis induced by the selected gauge-fixed potential basis through \(K_h\).
Let the matrices \(\mathcal G_h\) and \(\mathcal J_{Z,h}\) be defined by
\(
\boldsymbol x^{\trans}\mathcal G_h\boldsymbol x=s_h(\boldsymbol z_h,\boldsymbol z_h)\) and
\(\boldsymbol x^{\trans}\mathcal J_{Z,h}\boldsymbol x=\sum_i s_h(\boldsymbol z_i,\boldsymbol z_i),
\)
respectively.
The injectivity of the selected potential representation,
together with the coercivity of $s_h$ on $\Z_h$,
makes $\mathcal G_h$ positive definite. Hence,
\[
\begin{aligned}
\sup_{\boldsymbol x\ne0}
\frac{\boldsymbol x^{\trans}\mathcal J_{Z,h}\boldsymbol x}
     {\boldsymbol x^{\trans}\mathcal G_h\boldsymbol x}=
\lambda_{\max}\!\left(
\mathcal G_h^{-1/2}
\mathcal J_{Z,h}
\mathcal G_h^{-1/2}
\right)<\infty.
\end{aligned}
\]
This bounds \eqref{eq:kernel-decomposition-constant}; allowing \(\bm z_{\Z}\neq\bm 0\) can
only lower its infimum.
\end{proof}

Together, these decompositions control the two contributions in
\eqref{eq:exactnearlysingularform}: \cref{lem:shear-decomposition} supplies shear
stability for general displacements, while \cref{lem:kernel-decomposition} gives
the nullspace condition
\begin{equation}
 \Z_h=(\V_H^{\Z}\cap\Z_h)
       +\sum_{i=1}^{N_p}(\V_i\cap\Z_h),
 \label{eq:nullspace-condition}
\end{equation}
with stability constant \(C_{\Z}(h,H)\).  This is the kernel-compatible
decomposition used in parameter-uniform subspace correction for nearly singular
operators \cite{leewuxuzikatanov2007}.

\subsection{The additive split operator and the hybrid method}

The preceding decompositions naturally bound an additive subspace correction, whereas
\cref{alg:twolevel} uses forward and backward patch sweeps around the split coarse
correction.  The next comparison shows that this symmetric hybrid ordering preserves
the relevant spectral bounds.

Let \(\mathcal P_U\), \(\mathcal P_{\Z}\), and \(\mathcal P_i\) be the
\(a_{\lambda,h}\)-orthogonal projections onto \(\V_H^U\), \(\V_H^{\Z}\), and
\(\V_i\), respectively.  Thus, for example,
\begin{equation*}
 \mathcal P_U=P_UA_U^{-1}P_U^{\trans}A_h,
 \quad \mathcal P_{\Z}=P_{\Z}A_{\Z}^{-1}P_{\Z}^{\trans}A_h, \quad
 \mathcal P_i=R_i^{\trans}A_i^{-1}R_iA_h.
\end{equation*}
Set
\begin{equation}
 \mathcal P_c=\mathcal P_U+\mathcal P_{\Z},
 \qquad
 \mathcal P_{\rm AS}=\mathcal P_c+\sum_{i=1}^{N_p}\mathcal P_i
                     =B_{\rm AS}A_h,
 \label{eq:additiveschwarz}
\end{equation}
where
\begin{equation*}
 B_{\rm AS}=P_UA_U^{-1}P_U^{\trans}
       +P_{\Z}A_{\Z}^{-1}P_{\Z}^{\trans}
       +\sum_{i=1}^{N_p}R_i^{\trans}A_i^{-1}R_i.
\end{equation*}
The sum \(\mathcal P_c=B_HA_h\) preserves the additive structure of the two coarse
corrections in \cref{alg:twolevel}. To compare their symmetric composition with
\eqref{eq:additiveschwarz}, for the coloring \(\{\mathcal I_j\}_{j=1}^{N_c}\)
introduced in Section \ref{sec:Vertex-patch}, set
\(
 \mathcal H_j=\sum_{i\in\mathcal I_j}\mathcal P_i.
\)
For a patch \(\omega_i\), let
\(\mathcal N(i)=\{a:\operatorname{supp}(A_{h,a})\cap\omega_i\neq\emptyset\}\)
and define the enlarged-patch energy
\begin{equation*}
 a_{\omega_i}(\bm v_h,\bm v_h)
 =\sum_{a\in\mathcal N(i)}a_{\lambda,a}(\bm v_h,\bm v_h).
\end{equation*}
For each color \(j\), define
\(
m_j(a)
:=
\#\{\,i\in\mathcal I_j: a\in\mathcal N(i)\,\}\) and
\(C_{\rm col}
:=
\max_{j,a} m_j(a).
\)
Thus \(m_j(a)\) counts the number of patches of color \(j\)
whose enlarged neighborhoods contain the vertex contribution \(a\).
The local representation \(A_h=\sum_a A_{h,a}\), with positive-semidefinite
vertex contributions of fixed stencil width, and the bounded overlap of one color give
\begin{equation}
 0\leq\mathcal H_j\leq C_{\rm col}I
 \quad\hbox{in the \(A_h\)-inner product}.
 \label{eq:color-upper-bound}
\end{equation}
Indeed, the variational characterization of the local projection gives
\begin{align*}
 \norm{\mathcal P_i \bm v_h}_{A_h}=\sup_{0\neq \bm w_i\in\V_i}
   \frac{a_{\lambda,h}(\bm v_h,\bm w_i)}{\norm{\bm w_i}_{A_h}}\leq a_{\omega_i}(\bm v_h,\bm v_h)^{1/2},
\end{align*}
which follows by applying Cauchy--Schwarz to the positive-semidefinite local
forms with indices in \(\mathcal N(i)\); all other local forms vanish on \(\bm w_i\). By this inequality, we have
\begin{equation*}
 \sum_{i\in\mathcal I_j}\norm{\mathcal P_i \bm v_h}_{A_h}^2
 \leq \sum_a m_j(a)a_{\lambda,a}(\bm v_h,\bm v_h)
 \leq C_{\rm col}\norm{\bm v_h}_{A_h}^2.
\end{equation*}
Moreover, since the \(\mathcal P_i\) are orthogonal projections,
\(
 (\mathcal H_j\bm v_h,\bm v_h)_{A_h}
 =\sum_{i\in\mathcal I_j}\norm{\mathcal P_i \bm v_h}_{A_h}^2,
\)
which proves \eqref{eq:color-upper-bound}.  
The constant is purely an overlap count;
it depends on the stencil, patch enlargement, and coloring, but not on \(h\) or the
Lam\'e parameters. When the same-color enlarged neighborhoods are disjoint,
\(C_{\rm col}=1\).
Together with contractivity of the two global orthogonal projections, this gives the
explicit additive upper bound
\begin{equation}
 \begin{split}
 a_{\lambda,h}(\mathcal P_{\rm AS}\bm v_h,\bm v_h)
 &=\norm{\mathcal P_U\bm v_h}_{A_h}^2
   +\norm{\mathcal P_{\Z}\bm v_h}_{A_h}^2
   +\sum_{j=1}^{N_c}\sum_{i\in\mathcal I_j}\norm{\mathcal P_i\bm v_h}_{A_h}^2\\
 &\leq (2+N_cC_{\rm col})\norm{\bm v_h}_{A_h}^2.
 \end{split}
 \label{eq:additive-upper-bound}
\end{equation}
For the three-dimensional patches used below, the one-dimensional cell-index
sets are \(\{i-1,i,i+1\}\), while the energy at vertex \(a\) involves
\(\{a-1,a\}\).  Intersection requires
\(i\in\{a-2,a-1,a,a+1\}\), which contains at most two indices of each
modulo-three color.  Hence \(C_{\rm col}\leq 2^3=8\), and
\(\omega_s=0.24<2/8\) is admissible.  In two dimensions the symmetric
four-cell intervals give five candidate indices; modulo-five coloring yields
\(C_{\rm col}=1\) and admits \(\omega_s=1\).  Boundary truncation only reduces
these counts.

Choose \(0<\omega_s<2/C_{\rm col}\) and define the forward error propagator
\begin{equation}
 E_S=(I-\omega_s\mathcal H_{N_c})\cdots
     (I-\omega_s\mathcal H_1).
 \label{eq:forwarderror}
\end{equation}
With the fixed coloring above, \(N_c\) and \(C_{\rm col}\) depend on the stencil
and patch geometry.  We use the \(a_{\lambda,h}\)-adjoint in comparing this
propagator with the symmetric preconditioner.

\begin{lemma}[Comparison of split additive and symmetric hybrid corrections]
\label{lem:hybridcomparison}
There are constants \(c_G,C_G>0\), depending only on \(N_c\),
\(C_{\rm col}\), and \(\omega_s\), such that
\begin{equation}
 c_G\,a_{\lambda,h}(\mathcal P_{\rm AS}\bm v_h,\bm v_h)
 \leq a_{\lambda,h}(B_{\rm TL}A_h\bm v_h,\bm v_h)
 \leq C_G\,a_{\lambda,h}(\mathcal P_{\rm AS}\bm v_h,\bm v_h)
 \label{eq:hybridcomparison}
\end{equation}
for every \(\bm v_h\in \bm U_h\).  The constants are independent of \(h\), \(\lambda\),
and \(\mu\) for fixed \(N_c\), \(C_{\rm col}\), and \(\omega_s\).
\end{lemma}

\begin{proof}
For \(\bm e_0=\bm v_h\), set
\(\bm c_j=\omega_s\mathcal H_j\bm e_{j-1}\) and
\(\bm e_j=\bm e_{j-1}-\bm c_j\).  Since
\(0\leq\mathcal H_j\leq C_{\rm col}I\) and \(\|\mathcal H_j\bm w\|_{A_h}^2
\le C_{\rm col}(\mathcal H_j\bm w,\bm w)_{A_h}
\), the colorwise energy decrease satisfies
\begin{equation}
 \omega_s(2-\omega_s C_{\rm col})
 \sum_{i\in\mathcal I_j}\norm{\mathcal P_i\bm e_{j-1}}_{A_h}^2
 \leq \norm{\bm e_{j-1}}_{A_h}^2-\norm{\bm e_j}_{A_h}^2\leq
 2\omega_s\sum_{i\in\mathcal I_j}
 \norm{\mathcal P_i\bm e_{j-1}}_{A_h}^2.
 \label{eq:color-energy-drop}
\end{equation}
Put \(d_S(\bm v_h)=\norm{\bm v_h}_{A_h}^2-\norm{E_S\bm v_h}_{A_h}^2\),
\(x_j=\norm{\mathcal H_j^{1/2}\bm e_{j-1}}_{A_h}\), and
\(y_j=\norm{\mathcal H_j^{1/2}\bm v_h}_{A_h}\). The identity
\(
\bm e_{j-1}
=
\bm v_h-\omega_s\sum_{\ell<j}\mathcal H_\ell\bm e_{\ell-1}
\),
together with its rearrangement, gives
\begin{equation}
 x_j\leq y_j+\omega_sC_{\rm col}\sum_{\ell<j}x_\ell,
 \qquad
 y_j\leq x_j+\omega_sC_{\rm col}\sum_{\ell<j}x_\ell.
 \label{eq:finite-color-recursion}
\end{equation}
Indeed, \(\norm{\mathcal H_j^{1/2}\mathcal H_\ell^{1/2}}_{A_h}
\leq C_{\rm col}\), and therefore
\[
 \norm{\mathcal H_j^{1/2}\mathcal H_\ell\bm e_{\ell-1}}_{A_h}
 \leq C_{\rm col}\norm{\mathcal H_\ell^{1/2}\bm e_{\ell-1}}_{A_h}
 =C_{\rm col}x_\ell.
\]
Finite induction over the fixed number of colors therefore gives
\begin{equation}
 K_c^{-1}\sum_{j=1}^{N_c}y_j^2
 \leq\sum_{j=1}^{N_c}x_j^2
 \leq K_c\sum_{j=1}^{N_c}y_j^2,
 \label{eq:finite-color-norm-equivalence}
\end{equation}
with \(K_c\) depending only on \(N_c\) and \(\omega_sC_{\rm col}\).  Since
\(y_j^2=\sum_{i\in\mathcal I_j}\norm{\mathcal P_i\bm v_h}_{A_h}^2\),
\eqref{eq:color-energy-drop} and \eqref{eq:finite-color-norm-equivalence} prove
\begin{align}
 C_c^{-1}\sum_{i=1}^{N_p}\norm{\mathcal P_i\bm v_h}_{A_h}^2
 &\leq d_S(\bm v_h)
 \leq C_c\sum_{i=1}^{N_p}\norm{\mathcal P_i\bm v_h}_{A_h}^2,
 \label{eq:sweepadditivecomparison}\\
 \norm{\bm v_h-E_S\bm v_h}_{A_h}^2&\leq C_E d_S(\bm v_h),
 \label{eq:sweepcorrectionbound}
\end{align}
where \(C_c\) and \(C_E\) depend only on \(N_c,C_{\rm col},\omega_s\).
Indeed, the second bound follows from
\(\norm{\sum_j\bm c_j}_{A_h}^2\leq N_c\sum_j\norm{\bm c_j}_{A_h}^2\),
\(\norm{\bm c_j}_{A_h}^2\leq\omega_s^2C_{\rm col}
(\mathcal H_j\bm e_{j-1},\bm e_{j-1})_{A_h}\), and the lower bound in
\eqref{eq:color-energy-drop}.

Let \(E_S^\dagger\) denote the \(A_h\)-adjoint of \(E_S\).  From
\eqref{eq:twoleveloperator},
\begin{equation}
 B_{\rm TL}A_h
 =I-E_S^\dagger(I-\mathcal P_c)E_S.
 \label{eq:hybriderroridentity}
\end{equation}
Since \(\mathcal P_U\) and \(\mathcal P_{\Z}\) are orthogonal projections,
\(\norm{\mathcal P_c}_{A_h}\leq2\), and
\begin{equation}
 a_{\lambda,h}(B_{\rm TL}A_h\bm v_h,\bm v_h)
 =d_S(\bm v_h)+\norm{\mathcal P_UE_S\bm v_h}_{A_h}^2
             +\norm{\mathcal P_{\Z}E_S\bm v_h}_{A_h}^2.
 \label{eq:hybridrayleigh}
\end{equation}
Write \(\bm r_h=\bm v_h-E_S\bm v_h\), and set
\begin{equation*}
 X_c(\bm w_h)=\norm{\mathcal P_U\bm w_h}_{A_h}^2
          +\norm{\mathcal P_{\Z}\bm w_h}_{A_h}^2.
\end{equation*}
Projection contractivity and \eqref{eq:sweepcorrectionbound} give
\begin{equation*}
 X_c(\bm v_h)\leq2X_c(E_S\bm v_h)+4C_Ed_S(\bm v_h),
 \qquad
 X_c(E_S\bm v_h)\leq2X_c(\bm v_h)+4C_Ed_S(\bm v_h).
\end{equation*}
The second inequality bounds the right-hand side of \eqref{eq:hybridrayleigh} by
\((1+4C_E)d_S(\bm v_h)+2X_c(\bm v_h)\).  Conversely, the first inequality and
\(X_c(E_S\bm v_h)\geq0\) give, for every \(\theta\in(0,1]\),
\begin{equation*}
 d_S(\bm v_h)+X_c(E_S\bm v_h)
 \geq(1-2\theta C_E)\,d_S(\bm v_h)+\frac{\theta}{2}X_c(\bm v_h),
\end{equation*}
and \(\theta=\min\{1,(4C_E)^{-1}\}\) makes both coefficients positive.
Together with \eqref{eq:sweepadditivecomparison}, the right-hand side of
\eqref{eq:hybridrayleigh} is therefore uniformly equivalent to
\begin{equation*}
 \sum_i\norm{\mathcal P_i\bm v_h}_{A_h}^2
 +\norm{\mathcal P_U\bm v_h}_{A_h}^2
 +\norm{\mathcal P_{\Z}\bm v_h}_{A_h}^2
 =a_{\lambda,h}(\mathcal P_{\rm AS}\bm v_h,\bm v_h).
\end{equation*}
This proves \eqref{eq:hybridcomparison}.
\end{proof}

The preceding lemma also proves symmetry and positivity.  Indeed,
\eqref{eq:hybriderroridentity} is self-adjoint in the \(A_h\)-inner product, and the patch spaces span \(\bm U_h\), so
the right-hand side of \eqref{eq:hybridrayleigh} is positive for every nonzero
\(\bm v_h\).  Hence \(B_{\rm TL}\) is symmetric positive definite and may be used with
conjugate gradients.

\subsection{Parameter-uniform condition-number bound}

Kernel stability controls the limiting subspace; the complementary component must be
controlled through the volumetric operator.  Let \(Q_h^D:=\range(D_h)\subseteq Q_h^v\), equipped with the Euclidean norm
associated with \eqref{eq:qh}, and define
\begin{equation}
 \beta_h=
 \inf_{0\neq q_h\in Q_h^D}\ 
 \sup_{0\neq \bm v_h\in \bm U_h}
 \frac{q_h^{\trans}D_h\bm v_h}
      {s_h(\bm v_h,\bm v_h)^{1/2}\norm{q_h}_2}.
 \label{eq:auxiliaryinfsup}
\end{equation}
Since \(D_h:\bm U_h\to Q_h^D\) is surjective and \(s_h\) is
coercive, finite dimensionality implies \(\beta_h>0\) on every fixed grid. The signed-incidence formula and \eqref{eq:shear-jump-equivalence} also give
\begin{equation}
 q_h(\bm v_h,\bm v_h)\leq C_Ds_h(\bm v_h,\bm v_h),
 \label{eq:div-bounded-by-shear}
\end{equation}
with \(C_D\) independent of \(h\).  Indeed,
\eqref{eq:da-half-face} and the uniform Cartesian scalings
\(|e|\simeq h^{d-1}\), \(|\D_a|\simeq h^d\), and \(h_e\simeq h\) imply
\begin{equation*}
 |(D_h\bm v_h)_a|^2
 \leq C h^{d-2}
 \sum_{e\in\F_h^{\mathrm{pr},1/2}, e\subset\D_a}
       |\jump{\bm v_h}_e|^2.
\end{equation*}
Each half-face occurs in only a dimension-dependent number of interaction regions.
Summing over \(a\), using
\(\norm{\jump{\bm v_h}}_{L^2(e)}^2=|e||\jump{\bm v_h}_e|^2\), and then applying
\eqref{eq:shear-jump-equivalence} proves \eqref{eq:div-bounded-by-shear} with a
mesh-independent constant.

\begin{theorem}[Robustness with respect to the Lam\'e ratio]
\label{thm:uniformcondition}
Under \cref{ass:discrete-stability}, fix the nested grid pair and coloring
defined above.  The preconditioner in \cref{alg:twolevel}, with
\(0<\omega_s<2/C_{\rm col}\), satisfies
\begin{equation}
 \kappa(B_{\rm TL}A_h)\leq C_{\rm TL}(h,H),
 \label{eq:uniformcondition}
\end{equation}
where \(C_{\rm TL}(h,H)\) depends on the coarsening ratio, patch enlargement,
finite-color constants, \(C_U\), \(C_{\Z}(h,H)\), and \(\beta_h^{-1}\), and is
independent of \(\lambda/\mu\).
Here \(\kappa(B_{\rm TL}A_h)\) denotes the ratio of the extreme eigenvalues of the
symmetric positive definite matrix
\(A_h^{1/2}B_{\rm TL}A_h^{1/2}\), which is similar to \(B_{\rm TL}A_h\).
\end{theorem}

\begin{proof}
Multiplication of the fine form and all its exact restrictions by the same positive
scalar leaves the preconditioned operator unchanged.  We may therefore study
\(s_{\lambda,h}+tq_h\), where \(t=\lambda/\mu\).  If \(0\leq t\leq1\),
\eqref{eq:slambdaequivalence} and \eqref{eq:div-bounded-by-shear} make this form
uniformly equivalent to \(s_h\).  The decomposition of
\cref{lem:shear-decomposition} and the additive upper bound
\eqref{eq:additive-upper-bound} then give uniform lower and upper spectral bounds for
\(B_{\rm AS}A_h\).

Let now \(t\geq1\), set \(\epsilon=t^{-1}\), and scale to the form
\(q_h+\epsilon s_{\lambda,h}\), whose kernel at \(\epsilon=0\) is \(\Z_h\).
Let \(\beta_{\lambda,h}\) be \eqref{eq:auxiliaryinfsup} with \(s_h\) replaced by
\(s_{\lambda,h}\).  From \eqref{eq:slambdaequivalence},
\begin{equation}
 \beta_{\lambda,h}\geq C_s^{-1/2}\beta_h>0.
 \label{eq:betalambda-lower}
\end{equation}
On the \(s_{\lambda,h}\)-orthogonal complement \(\bm W_{\lambda,h}\) of \(\Z_h\),
the singular-value characterization of \(\beta_{\lambda,h}\) gives
\begin{equation}
 s_{\lambda,h}(\bm w_h,\bm w_h)
 \leq\beta_{\lambda,h}^{-2}q_h(\bm w_h,\bm w_h)
 \qquad\forall \bm w_h\in \bm W_{\lambda,h}.
 \label{eq:complement-control}
\end{equation}
For \(\bm v_h=\bm z_h+\bm w_h\), with \(\bm z_h\in\Z_h\) and
\(\bm w_h\in\bm W_{\lambda,h}\), apply \cref{lem:kernel-decomposition} to \(\bm z_h\)
and \cref{lem:shear-decomposition} to \(\bm w_h\):
\begin{equation*}
 \bm z_h=\bm z_{\Z}+\sum_i \bm z_i,
 \qquad \bm w_h=\bm w_U+\sum_i\bm w_i,
\end{equation*}
where \(\bm z_{\Z}\in\V_H^{\Z}\), \(\bm z_i\in\V_i\cap\Z_h\),
\(\bm w_U\in\V_H^U\), and \(\bm w_i\in\V_i\).  
Since the decomposition
\(\bm v_h=\bm z_h+\bm w_h\) is
\(s_{\lambda,h}\)-orthogonal and
\(\bm z_h\in\Z_h\), we have
\[
q_h(\bm v_h,\bm v_h)=q_h(\bm w_h,\bm w_h),
\quad
s_{\lambda,h}(\bm v_h,\bm v_h)
=
s_{\lambda,h}(\bm z_h,\bm z_h)
+
s_{\lambda,h}(\bm w_h,\bm w_h).
\]
Moreover, since
\(\bm z_{\Z},\bm z_i\in\Z_h\),
we have \(q_h(\bm z_{\Z},\bm z_{\Z})=0\).
Setting \(\bm v_i=\bm z_i+\bm w_i\) gives
\(q_h(\bm v_i,\bm v_i)=q_h(\bm w_i,\bm w_i)\).
The two stability estimates in
\cref{lem:shear-decomposition,lem:kernel-decomposition}, together with
\eqref{eq:slambdaequivalence}, \eqref{eq:div-bounded-by-shear}, and
\eqref{eq:complement-control}, therefore yield
\begin{align}
 &\epsilon s_{\lambda,h}(\bm z_{\Z},\bm z_{\Z})
 +\bigl(q_h(\bm w_U,\bm w_U)+\epsilon s_{\lambda,h}(\bm w_U,\bm w_U)\bigr)
 +\sum_i\bigl(q_h(\bm v_i,\bm v_i)+\epsilon s_{\lambda,h}(\bm v_i,\bm v_i)\bigr)
 \notag\\
 &\hspace{2cm}\leq
 C(h,H)\bigl(q_h(\bm v_h,\bm v_h)+\epsilon s_{\lambda,h}(\bm v_h,\bm v_h)\bigr),
 \qquad 0<\epsilon\leq1.
 \label{eq:full-energy-stable-decomposition}
\end{align}
The constant depends on \(C_U\), \(C_{\Z}(h,H)\), \(\beta_h^{-1}\), and the
equivalence constants, but not on \(\epsilon\).  The exact-solver additive Schwarz
identity \cite{xu1992,xuzikatanov2002} therefore combines \eqref{eq:full-energy-stable-decomposition}
with \eqref{eq:additive-upper-bound} to yield
\begin{equation}
 \kappa(B_{\rm AS}A_h)\leq C_{\rm AS}(h,H),
 \label{eq:additivebound}
\end{equation}
independently of \(t\).  Finally, \cref{lem:hybridcomparison} transfers this bound
to the symmetric hybrid operator and proves \eqref{eq:uniformcondition}.
\end{proof}

The standard PCG energy-error estimate consequently gives a contraction factor
controlled by \(C_{\rm TL}(h,H)\) uniformly in the Lam\'e ratio.

\section{Numerical experiments}
\label{sec:numerics}

We present two- and three-dimensional examples to demonstrate the effectiveness of
the two-subspace preconditioner under mesh refinement and increasing
incompressibility.  In each example, an exact solution relates the solver
accuracy to the approximation error, and a comparison with the displacement
coarse space alone identifies the contribution of the kernel correction.
Material interfaces and nonhomogeneous displacement data introduce variations
in the coefficients and boundary conditions within this common framework.

We use conjugate gradients with the stopping criterion
\(\norm{A_h\boldsymbol u_k-\boldsymbol b_h}_2/
\norm{\boldsymbol b_h}_2\leq10^{-7}\).
The reported PC count is the number of applications of the symmetric
composition in \eqref{eq:twoleveloperator}, with one forward and one backward
vertex-patch sweep.  The grid ratio is \(H/h=8\), and both coarse systems are
solved by direct factorization.  We take \(\omega_s=1\) in two dimensions.
In three dimensions, modulo-three coloring in each
coordinate direction gives 27 colors.  We take \(\omega_s=0.24\), which
satisfies the sufficient damping condition
\(0<\omega_s<1/4\) obtained from the overlap bound in
\cref{lem:hybridcomparison}.

Patch factorizations are computed during setup and reused throughout the
iteration.  The PETSc implementation batches same-color patch solves and
uses MUMPS for coarse LU factorizations.  Parallel computations run on
Intel Xeon Platinum 8358 processors, with one MPI rank per active core;
the two-dimensional interface example uses a MATLAB implementation of the
same subspace corrections.
The setup time \(t_{\rm setup}\) includes fine-grid assembly, construction of
both transfers and Galerkin matrices, and all patch and coarse factorizations;
the solve time \(t_{\rm solve}\) contains the Krylov iteration and its
preconditioner applications.  Times are reported in seconds.

For each nonzero exact field
\(q\in\{\underline{\sigma},\operatorname{div}\underline{\sigma},
\bm u,\underline{\gamma}\}\), we measure the relative error
\begin{equation}
 E_q=\frac{\norm{q-q_h}_{L^2(\Omega)}}{\norm{q}_{L^2(\Omega)}}.
 \label{eq:p0error}
\end{equation}
The integrals are evaluated on interaction subcells for stress and rotation,
and on macro-elements for stress divergence and displacement.  All reported
errors are computed from the iterative solutions with the stopping criterion
above.

\subsection{A smooth nearly incompressible solution}
\label{sec:numerics-2d}

We begin with the smooth manufactured solution from
\cite[Section~3.1]{valseth2021} on \(\Omega=(0,1)^2\), with homogeneous
displacement data and
\begin{equation}
 u_1(x,y)=u_2(x,y)=\sin(\pi x)\sin(\pi y).
 \label{eq:valsethsolution}
\end{equation}
With \(\mu=1\), the body force is determined from the first two equations in
\eqref{eq:strong-mixed}.  Fixing \(\lambda/\mu=10^8\),
\cref{tab:smooth2dconv} reports the errors of the displacement and all locally
recovered fields.

\begin{table}[H]
  \centering
  \caption{Relative errors for the smooth two-dimensional problem with
  \(\lambda/\mu=10^8\) and \(H/h=8\).}
  \label{tab:smooth2dconv}
  \setlength{\tabcolsep}{5.0pt}
  \begin{tabular}{cc*{4}{cc}}
    \toprule
    \multirow{2}{*}{grid} & \multirow{2}{*}{PC calls} &
    \multicolumn{2}{c}{\(\underline{\sigma}\)} &
    \multicolumn{2}{c}{\(\operatorname{div}\underline{\sigma}\)} &
    \multicolumn{2}{c}{\(\bm u\)} &
    \multicolumn{2}{c}{\(\underline{\gamma}\)} \\
    \cmidrule(lr){3-4}\cmidrule(lr){5-6}
    \cmidrule(lr){7-8}\cmidrule(lr){9-10}
    & & error & rate & error & rate & error & rate & error & rate \\
    \midrule
    \(32^2\) & 15 & \(4.01\mathrm{e}{-2}\) & / & \(4.01\mathrm{e}{-2}\) & / & \(4.02\mathrm{e}{-2}\) & / & \(4.37\mathrm{e}{-2}\) & / \\
    \(64^2\) & 20 & \(2.00\mathrm{e}{-2}\) & 1.0 & \(2.00\mathrm{e}{-2}\) & 1.0 & \(2.01\mathrm{e}{-2}\) & 1.0 & \(2.10\mathrm{e}{-2}\) & 1.1 \\
    \(128^2\) & 28 & \(1.00\mathrm{e}{-2}\) & 1.0 & \(1.00\mathrm{e}{-2}\) & 1.0 & \(1.00\mathrm{e}{-2}\) & 1.0 & \(1.03\mathrm{e}{-2}\) & 1.0 \\
    \(256^2\) & 28 & \(5.01\mathrm{e}{-3}\) & 1.0 & \(5.01\mathrm{e}{-3}\) & 1.0 & \(5.01\mathrm{e}{-3}\) & 1.0 & \(5.07\mathrm{e}{-3}\) & 1.0 \\
    \(512^2\) & 28 & \(2.50\mathrm{e}{-3}\) & 1.0 & \(2.50\mathrm{e}{-3}\) & 1.0 & \(2.51\mathrm{e}{-3}\) & 1.0 & \(2.52\mathrm{e}{-3}\) & 1.0 \\
    \bottomrule
  \end{tabular}
\end{table}

The errors decrease at first order through \(512^2\), including the stress
and rotation recovered by local back-substitution.  The PC count increases
from 15 to 28 over the same sequence, so the refinement needed to improve
accuracy produces only a moderate increase in iteration count.

To distinguish mesh dependence from material-parameter dependence,
\cref{tab:lambdarobust2d} varies the Lam\'e ratio on each grid.  The first
five PC columns use both coarse subspaces, whereas the final column retains
only the conventional displacement correction on the \(128^2\) grid.

\begin{table}[htbp]
  \centering
  \caption{Preconditioner applications for the smooth two-dimensional problem with
  \(H/h=8\).  The last column uses only the displacement coarse space on the
  \(128^2\) grid.}
  \label{tab:lambdarobust2d}
  \setlength{\tabcolsep}{6pt}
  \begin{tabular}{cc*{5}{c}c}
    \toprule
    \multirow{2}{*}{\(\lambda/\mu\)} & \multirow{2}{*}{\(\nu\)} &
    \multicolumn{5}{c}{kernel-enriched PC calls} &
    \(\text{disp.-only}\) \\
    \cmidrule(lr){3-7}
      & & \(32^2\) & \(64^2\) & \(128^2\) & \(256^2\) & \(512^2\) & \(128^2\) \\
    \midrule
    \(1\) & \(0.2500000\) & 11 & 13 & 14 & 15 & 15 & 14 \\
    \(10^2\) & \(0.4950495\) & 19 & 22 & 25 & 27 & 27 & 27 \\
    \(10^4\) & \(0.4999500\) & 19 & 23 & 25 & 26 & 28 & 82 \\
    \(10^6\) & \(0.4999995\) & 15 & 22 & 29 & 27 & 30 & 190 \\
    \(10^8\) & \(0.499999995\) & 15 & 20 & 28 & 28 & 28 & 197 \\
    \bottomrule
  \end{tabular}
\end{table}

At each mesh resolution, the two-subspace method maintains low iteration
counts throughout the parameter range, whereas the displacement correction
alone becomes increasingly expensive near incompressibility.
On the \(128^2\) grid, both choices require 14 applications at
\(\lambda=\mu\), but at \(\lambda/\mu=10^8\) the enriched method requires
28 applications compared with 197 for the displacement correction alone.
The additional coarse space therefore becomes particularly effective as the
volumetric constraint dominates.

\subsection{Discontinuous material coefficients}
\label{sec:numerics-interface}

To combine near incompressibility with a material interface, we use the stiff
inclusion problem from \cite[Example~3]{fuzhao2025}.  On \(\Omega=(0,1)^2\),
let \(D=(1/3,2/3)^2\) and set
\begin{equation*}
 a(x,y)=\begin{cases}10^4,&(x,y)\in D,\\1,&(x,y)\notin D,\end{cases}
 \qquad \mu=a,\qquad \lambda=\rho a.
\end{equation*}
With homogeneous displacement data, the exact solution is
\begin{equation*}
 u_1(x,y)=u_2(x,y)=a(x,y)^{-1}\sin(3\pi x)\sin(3\pi y).
\end{equation*}
The trigonometric factor vanishes on \(\partial D\), so the displacement is
continuous across the interface.  The common coefficient factor in
\(\lambda\) and \(\mu\) also makes the exact stress and the scaled rotation
\(\underline{\widetilde\gamma}=2\mu\underline\gamma\) independent of \(a\).

Accordingly, we approximate \(\underline{\widetilde\gamma}_h\) by a constant
on each interaction region and use the weighted rotation space
\(\underline\Gamma_h^\mu=(2\mu)^{-1}\underline\Gamma_h\) in both the
constitutive and weak-symmetry equations.  This choice represents the jump
in the physical rotation through
\(\underline\gamma_h=(2\mu)^{-1}\underline{\widetilde\gamma}_h\), while
preserving symmetry and vertex-local elimination.  The reported rotation
error is measured for \(\underline\gamma_h\).

For this interface-fitted problem, \cref{tab:highcontrast-convergence}
reports convergence at \(\lambda/\mu=10^8\), and
\cref{tab:highcontrast-parameters} compares the two coarse-space choices
over the full parameter range.

\begin{table}[htbp]
  \centering
  \caption{Relative errors for the high-contrast problem with
  \(\lambda/\mu=10^8\), coefficient contrast \(10^4\), and \(H/h=8\).}
  \label{tab:highcontrast-convergence}
  \setlength{\tabcolsep}{5.0pt}
  \begin{tabular}{cc*{4}{cc}}
    \toprule
    \multirow{2}{*}{grid} & \multirow{2}{*}{PC calls} &
    \multicolumn{2}{c}{\(\underline{\sigma}\)} &
    \multicolumn{2}{c}{\(\operatorname{div}\underline{\sigma}\)} &
    \multicolumn{2}{c}{\(\bm u\)} &
    \multicolumn{2}{c}{\(\underline{\gamma}\)} \\
    \cmidrule(lr){3-4}\cmidrule(lr){5-6}
    \cmidrule(lr){7-8}\cmidrule(lr){9-10}
    & & error & rate & error & rate & error & rate & error & rate \\
    \midrule
    \(24^2\) & 14 & \(1.64\mathrm{e}{-1}\) & / & \(1.60\mathrm{e}{-1}\) & / & \(1.67\mathrm{e}{-1}\) & / & \(1.94\mathrm{e}{-1}\) & / \\
    \(48^2\) & 23 & \(8.06\mathrm{e}{-2}\) & 1.0 & \(8.01\mathrm{e}{-2}\) & 1.0 & \(8.10\mathrm{e}{-2}\) & 1.0 & \(8.85\mathrm{e}{-2}\) & 1.1 \\
    \(96^2\) & 35 & \(4.01\mathrm{e}{-2}\) & 1.0 & \(4.01\mathrm{e}{-2}\) & 1.0 & \(4.02\mathrm{e}{-2}\) & 1.0 & \(4.21\mathrm{e}{-2}\) & 1.1 \\
    \(192^2\) & 42 & \(2.00\mathrm{e}{-2}\) & 1.0 & \(2.00\mathrm{e}{-2}\) & 1.0 & \(2.01\mathrm{e}{-2}\) & 1.0 & \(2.05\mathrm{e}{-2}\) & 1.0 \\
    \bottomrule
  \end{tabular}
\end{table}

\begin{table}[htbp]
  \centering
  \caption{Preconditioner applications for the high-contrast problem.  The
  coefficient contrast is \(10^4\) and \(H/h=8\).  The last column uses only
  the displacement coarse space on the \(96^2\) grid.}
  \label{tab:highcontrast-parameters}
  \setlength{\tabcolsep}{6pt}
  \begin{tabular}{cc*{4}{c}c}
    \toprule
    \multirow{2}{*}{\(\lambda/\mu\)} & \multirow{2}{*}{\(\nu\)} &
    \multicolumn{4}{c}{kernel-enriched PC calls} & \(\text{disp.-only}\) \\
    \cmidrule(lr){3-6}
      & & \(24^2\) & \(48^2\) & \(96^2\) & \(192^2\) & \(96^2\) \\
    \midrule
    \(1\) & \(0.2500000\) & 12 & 16 & 17 & 19 & 16 \\
    \(10^2\) & \(0.4950495\) & 14 & 22 & 26 & 27 & 48 \\
    \(10^4\) & \(0.4999500\) & 17 & 24 & 28 & 31 & 115 \\
    \(10^6\) & \(0.4999995\) & 21 & 34 & 33 & 43 & 239 \\
    \(10^8\) & \(0.499999995\) & 14 & 23 & 35 & 42 & 350 \\
    \bottomrule
  \end{tabular}
\end{table}

The interface-fitted refinement preserves first-order convergence for all
four fields in \cref{tab:highcontrast-convergence}, including the physical
rotation recovered from the weighted space.  The material jump also leaves
the kernel correction effective in the nearly incompressible regime:
on the \(96^2\) grid at \(\lambda/\mu=10^8\), the two-subspace method
requires 35 applications, compared with 350 for the displacement correction
alone.  This contrast persists across the parameter range in
\cref{tab:highcontrast-parameters}: the two-subspace counts remain
controlled as incompressibility increases, while the displacement-only
counts rise from 16 to 350.

\subsection{A homogeneous oscillatory solution in three dimensions}
\label{sec:numerics-3d}

The next experiment extends the accuracy and robustness study to large
three-dimensional displacement systems.  On \(\Omega=(0,1)^3\), we prescribe
homogeneous Dirichlet data and take
\begin{equation}
 \bm u(x,y,z)=\lambda^{-1}\sin(3\pi x)\sin(3\pi y)\sin(3\pi z)
 \begin{pmatrix}1\\1\\1\end{pmatrix}.
 \label{eq:exactsine3d}
\end{equation}
Young's modulus is \(E=1\), with
\(\mu=E/[2(1+\nu)]\) and
\(\lambda=E\nu/[(1+\nu)(1-2\nu)]\).
The body force follows from \eqref{eq:strong-mixed}.  The factor
\(\lambda^{-1}\) keeps the volumetric stress and load bounded as
\(\lambda/\mu\) increases.  Since the exact rotation is nonzero, this
solution allows the approximation of all four mixed fields to be measured
in three dimensions.

At \(\lambda/\mu=10^8\), \cref{tab:exactsine3dconv} reports the complete
field errors together with the iteration count and computational cost.

\begin{table}[htbp]
  \centering
  \caption{Three-dimensional convergence with \(\lambda/\mu=10^8\), \(H/h=8\),
  and \(\omega_s=0.24\).  Times are in seconds.}
  \label{tab:exactsine3dconv}
  \footnotesize
  \setlength{\tabcolsep}{3.8pt}
  \begin{tabular}{ccccc*{4}{cc}}
    \toprule
    \multirow{2}{*}{grid} & \multirow{2}{*}{PC calls} &
    \multirow{2}{*}{cores} & \multirow{2}{*}{\(t_{\rm setup}\)} &
    \multirow{2}{*}{\(t_{\rm solve}\)} &
    \multicolumn{2}{c}{\(\underline{\sigma}\)} &
    \multicolumn{2}{c}{\(\operatorname{div}\underline{\sigma}\)} &
    \multicolumn{2}{c}{\(\bm u\)} &
    \multicolumn{2}{c}{\(\underline{\gamma}\)} \\
    \cmidrule(lr){6-7}\cmidrule(lr){8-9}\cmidrule(lr){10-11}\cmidrule(lr){12-13}
    & & & & & error & rate & error & rate & error & rate & error & rate \\
    \midrule
    \(16^3\) & 10 & 8 & 0.11 & 0.20 & \(3.26\mathrm{e}{-1}\) & / & \(2.92\mathrm{e}{-1}\) & / & \(3.35\mathrm{e}{-1}\) & / & \(3.62\mathrm{e}{-1}\) & / \\
    \(32^3\) & 16 & 8 & 0.67 & 1.21 & \(1.51\mathrm{e}{-1}\) & 1.1 & \(1.47\mathrm{e}{-1}\) & 1.0 & \(1.53\mathrm{e}{-1}\) & 1.1 & \(1.64\mathrm{e}{-1}\) & 1.1 \\
    \(64^3\) & 23 & 8 & 5.17 & 12.40 & \(7.41\mathrm{e}{-2}\) & 1.0 & \(7.36\mathrm{e}{-2}\) & 1.0 & \(7.44\mathrm{e}{-2}\) & 1.0 & \(7.72\mathrm{e}{-2}\) & 1.1 \\
    \(128^3\) & 26 & 64 & 7.45 & 18.01 & \(3.69\mathrm{e}{-2}\) & 1.0 & \(3.68\mathrm{e}{-2}\) & 1.0 & \(3.69\mathrm{e}{-2}\) & 1.0 & \(3.76\mathrm{e}{-2}\) & 1.0 \\
    \(256^3\) & 29 & 128 & 37.09 & 88.27 & \(1.84\mathrm{e}{-2}\) & 1.0 & \(1.84\mathrm{e}{-2}\) & 1.0 & \(1.84\mathrm{e}{-2}\) & 1.0 & \(1.86\mathrm{e}{-2}\) & 1.0 \\
    \bottomrule
  \end{tabular}
\end{table}

The refinement sequence in \cref{tab:exactsine3dconv} retains first-order
accuracy for all four fields as the displacement system grows to
50,331,648 unknowns.  Over the same sequence, the PC count increases from
10 to 29.  At the largest scale, setup and solution take 37.09 and
88.27 seconds, respectively, on 128 cores.

To examine parameter robustness on the same grid sequence,
\cref{tab:exactsine3dparam} varies \(\lambda/\mu\) and includes a
displacement-only comparison on the \(64^3\) grid.

\begin{table}[htbp]
  \centering
  \caption{Preconditioner applications for the homogeneous three-dimensional
  problem with \(H/h=8\) and \(\omega_s=0.24\).  The last column uses only the
  displacement coarse space on the \(64^3\) grid.}
  \label{tab:exactsine3dparam}
  \setlength{\tabcolsep}{5.2pt}
  \begin{tabular}{cc*{5}{c}c}
    \toprule
    \multirow{2}{*}{\(\lambda/\mu\)} & \multirow{2}{*}{\(\nu\)} &
    \multicolumn{5}{c}{kernel-enriched PC calls} & \(\text{disp.-only}\) \\
    \cmidrule(lr){3-7}
      & & \(16^3\) & \(32^3\) & \(64^3\) & \(128^3\) & \(256^3\) & \(64^3\) \\
    \midrule
    \(1\) & \(0.2500000\) & 6 & 10 & 12 & 14 & 15 & 12 \\
    \(10^2\) & \(0.4950495\) & 13 & 22 & 28 & 32 & 34 & 27 \\
    \(10^4\) & \(0.4999500\) & 15 & 29 & 35 & 42 & 49 & 92 \\
    \(10^6\) & \(0.4999995\) & 12 & 22 & 30 & 34 & 36 & 159 \\
    \(10^8\) & \(0.499999995\) & 10 & 16 & 23 & 26 & 29 & 135 \\
    \bottomrule
  \end{tabular}
\end{table}

The parameter sweep in \cref{tab:exactsine3dparam} shows that the
two-subspace iteration counts do not grow persistently with the Lam\'e
ratio.  On the \(256^3\) grid, the largest count is 49 at
\(\lambda/\mu=10^4\), decreasing to 29 at \(10^8\).
The comparison on the \(64^3\) grid identifies the role of the kernel
correction: both methods require 12 applications at \(\lambda=\mu\),
whereas at \(\lambda/\mu=10^6\) the additional correction reduces the
count from 159 to 30.  Its benefit thus becomes substantial as the
volumetric penalty increases, complementing the conventional displacement
space with nearly volume-preserving error components.

\subsection{Nonhomogeneous displacement data in three dimensions}
\label{sec:numerics-boundary}

To extend the comparison from a fixed boundary to prescribed boundary
deformation, we use the unit-cube benchmark of Carstensen and
Heuer~\cite{carstensenheuer2025}.  With \(E=1\), the exact displacement is
\begin{equation}
 \bm u(x,y,z)=
 \begin{pmatrix}
   \sin(3x)\cos(3y)\cos(3z)\\
   \cos(3x)\sin(3y)\cos(3z)\\
   \cos(3x)\cos(3y)\sin(3z)
 \end{pmatrix},
 \qquad \bm u=\bm u_D\quad\hbox{on }\partial\Omega.
 \label{eq:carstensenheuer3d}
\end{equation}
The load is \(\bm f=27(\lambda+2\mu)\bm u\).  The exact rotation vanishes,
so the convergence table reports the relative errors of stress, stress
divergence, and displacement.

With the prescribed boundary values included in the load, we use the same
preconditioner and report approximation errors in
\cref{tab:nonhomogeneous3d-convergence} at
\(\lambda/\mu=10^8\), and \cref{tab:nonhomogeneous3d-parameters}
compares the coarse-space choices as the material approaches
incompressibility.

\begin{table}[H]
  \centering
  \caption{Three-dimensional convergence for nonhomogeneous Dirichlet data with
  \(\lambda/\mu=10^8\), \(H/h=8\), and \(\omega_s=0.24\).  Times are in seconds.}
  \label{tab:nonhomogeneous3d-convergence}
  \footnotesize
  \setlength{\tabcolsep}{2.3pt}
  \begin{tabular}{ccccc*{3}{cc}}
    \toprule
    \multirow{2}{*}{grid} & \multirow{2}{*}{PC calls} &
    \multirow{2}{*}{cores} & \multirow{2}{*}{\(t_{\rm setup}\)} &
    \multirow{2}{*}{\(t_{\rm solve}\)} &
    \multicolumn{2}{c}{\(\underline{\sigma}\)} &
    \multicolumn{2}{c}{\(\operatorname{div}\underline{\sigma}\)} &
    \multicolumn{2}{c}{\(\bm u\)} \\
    \cmidrule(lr){6-7}\cmidrule(lr){8-9}\cmidrule(lr){10-11}
    & & & & & error & rate & error & rate & error & rate \\
    \midrule
    \(16^3\) & 12 & 8 & 0.11 & 0.24 & \(9.93\mathrm{e}{-2}\) & / & \(9.53\mathrm{e}{-2}\) & / & \(9.86\mathrm{e}{-2}\) & / \\
    \(32^3\) & 15 & 8 & 0.67 & 1.12 & \(4.92\mathrm{e}{-2}\) & 1.0 & \(4.77\mathrm{e}{-2}\) & 1.0 & \(4.81\mathrm{e}{-2}\) & 1.0 \\
    \(64^3\) & 17 & 8 & 5.17 & 9.17 & \(2.46\mathrm{e}{-2}\) & 1.0 & \(2.38\mathrm{e}{-2}\) & 1.0 & \(2.40\mathrm{e}{-2}\) & 1.0 \\
    \(128^3\) & 20 & 64 & 7.11 & 13.92 & \(1.23\mathrm{e}{-2}\) & 1.0 & \(1.19\mathrm{e}{-2}\) & 1.0 & \(1.20\mathrm{e}{-2}\) & 1.0 \\
    \(256^3\) & 21 & 128 & 37.21 & 64.42 & \(6.14\mathrm{e}{-3}\) & 1.0 & \(5.96\mathrm{e}{-3}\) & 1.0 & \(6.11\mathrm{e}{-3}\) & 1.0 \\
    \bottomrule
  \end{tabular}
\end{table}

\begin{table}[H]
  \centering
  \caption{Preconditioner applications for the nonhomogeneous Dirichlet problem.
  Here \(H/h=8\) and \(\omega_s=0.24\).
  The last column uses only the displacement coarse space on the \(64^3\) grid.}
  \label{tab:nonhomogeneous3d-parameters}
  \setlength{\tabcolsep}{5.2pt}
  \begin{tabular}{cc*{5}{c}c}
    \toprule
    \multirow{2}{*}{\(\lambda/\mu\)} & \multirow{2}{*}{\(\nu\)} &
    \multicolumn{5}{c}{kernel-enriched PC calls} & \(\text{disp.-only}\) \\
    \cmidrule(lr){3-7}
      & & \(16^3\) & \(32^3\) & \(64^3\) & \(128^3\) & \(256^3\) & \(64^3\) \\
    \midrule
    \(1\) & \(0.2500000\) & 7 & 9 & 11 & 12 & 12 & 11 \\
    \(10^2\) & \(0.4950495\) & 15 & 21 & 25 & 26 & 28 & 23 \\
    \(10^4\) & \(0.4999500\) & 20 & 27 & 32 & 37 & 39 & 73 \\
    \(10^6\) & \(0.4999995\) & 14 & 21 & 25 & 29 & 25 & 144 \\
    \(10^8\) & \(0.499999995\) & 12 & 15 & 17 & 20 & 21 & 98 \\
    \bottomrule
  \end{tabular}
\end{table}

The three fields retain approximately first-order convergence through
\(256^3\) with the prescribed boundary deformation.  At
\(\lambda/\mu=10^8\), this refinement increases the enriched PC count
from 12 to 21.  The parameter sweep in
\cref{tab:nonhomogeneous3d-parameters} exhibits the same behavior as the
homogeneous-boundary example: the counts peak at an intermediate Lam\'e
ratio and remain controlled toward the incompressible limit.
On the \(64^3\) grid, both coarse-space choices use
11 applications when \(\lambda=\mu\), whereas at \(\lambda/\mu=10^6\)
the kernel correction reduces the count from 144 to 25.  Its contribution
in the nearly incompressible regime is therefore retained when nonzero
boundary moments enter the reduced load.

\section{Concluding remarks}

The proposed two-grid preconditioner combines the compactness of local mixed-variable
elimination with a coarse correction that preserves the discrete incompressibility
constraint.  A potential-generated displacement space complements the conventional
geometric coarse space, while symmetric vertex-patch sweeps resolve local coupling.
The reduced energy decomposition links this construction to parameter-uniform
subspace-correction estimates, and the numerical comparisons demonstrate the
substantial reduction in parameter sensitivity provided by the kernel correction.
The smooth-solution experiments retain first-order approximation in the nearly incompressible
regime and extend to more than fifty million displacement unknowns.  The
potential-based construction also provides a starting point for multilevel
extensions and compatible transfers on more general mesh hierarchies.

\end{document}